\documentclass{siamart190516}
\usepackage{amssymb,amsmath}
\usepackage{bm,graphicx,epstopdf}
\usepackage{dirtytalk}
\usepackage{booktabs}
\renewcommand{\hat}[1]{\widehat{#1}}

\usepackage[noend]{algpseudocode}
\newcommand{\RR}{\mathbb R}

\newcommand{\trace}{{\rm trace}}

\definecolor{greenf}{rgb}{.054, .5, .005}

\newtheorem{remark}[theorem]{Remark}

\begin{document}
%\title{Preconditioned log-determinant approximation:\\ one Gaussian probe vector is almost always enough
%Nel seguente ordine la mia top3:
\title{Detecting when one probe vector is enough for preconditioned log-determinant approximation
%\title{On budget allocation in preconditioned log-determinant approximation
%\title{Preconditioned log-determinant approximation: when is one sample enough? 
\thanks{Version of \today.} }{}
\date{\today}
\author{Alice Cortinovis\thanks{Department of Computer Science, University of Pisa, PI, Italy. Email: {\tt  alice.cortinovis@unipi.it}} \and
 Daniele Toni\thanks{Scuola Normale Superiore di Pisa, PI, Italy. Email: {\tt  daniele.toni@sns.it}
}
}

\maketitle

\begin{abstract}
We present randomized algorithms for estimating the log-determinant of regularized symmetric positive semi-definite matrices. The algorithms access the matrix only through matrix vector products, and are based on the introduction of a preconditioner and stochastic trace estimator.
We claim that preconditioning as much as we can and making a rough estimate of the residual part with a small budget achieves a small error in most of the cases. We choose a Nystr\"om preconditioner and estimate the residual using only one sample of stochastic Lanczos quadrature (SLQ). We analyze the performance of this strategy from a theoretical and practical viewpoint. We also present an algorithm that, at almost no additional cost, detects whether the proposed strategy is not the most effective, in which case it uses more samples for the SLQ part. Numerical examples on several test matrices show that our proposed methods are competitive with existing algorithms.
\end{abstract}



%%%%%%%%%%%%%%%%%%%%%%%%%%%%%%%%%%%%%%%%%%%%%%%%%%%

\section{Introduction}
Given a symmetric positive semidefinite (SPSD) matrix $A \in \RR^{n\times n}$, we are interested in the computation of its regularized log-determinant
\begin{equation}\label{eq:logdet}
    \log\det(A+ I).
\end{equation}
This quantity arises in many applications: for instance, in machine learning and in statistics, when dealing with the discretization of covariance kernels, see e.g.~\cite{RasmussenWilliams06,Seeger04}, and it plays a central role in the log-marginal likelihood for Gaussian processes~\cite{CoverThomas06,GardnerPleissWeinbergerBindelWilson18}, which appears in hyperparameter optimization~\cite{WengerPleissHennigCunninghamGardner22} and in the normalization of the determinantal point processes for supervised learning~\cite{KuleszaTaskar12}. 

In principle,~\eqref{eq:logdet} can be expensive to compute, requiring up to $\mathcal O(n^3)$ operations. A popular strategy to \emph{approximate} this quantity in the case of a numerically low-rank matrix $A$ is to find a rank-$\ell$ approximation $S$ of $A$ and compute
\begin{equation}
	\label{eq:lowrank}
	\tag{low-rank}
	\log\det(A+I) \approx \log\det(S + I);
\end{equation}
see, e.g.~\cite{LiZhu21,PerssonKressner23,PerssonMeyerMusco25,SaibabaAlexanderianIpsen17}. Other strategies are based on the relation 
\begin{equation*}
	\log \det(A + I) = \trace \log(A + I),
\end{equation*}
where $\log(A+I)$ is intended as a matrix function~\cite{Higham08}. This expression is the basis for stochastic Lanczos quadrature (SLQ)~\cite{UbaruChenSaad17}, which combines a randomized method to estimate the trace of a matrix with a Krylov subspace method to approximate quadratic forms involving a matrix function. More specifically, for the symmetric matrix $\log(A+I)$, given $N$ isotropic independently drawn random vectors  $w_1, \ldots, w_N$, the Girard-Hutchinson estimator~\cite{Girard89,Hutchinson90} constructs an unbiased approximation of the trace as
\begin{equation} \label{eq:gh}
\tag{SLQ}
	\trace\log(A+I) \approx \frac{1}{N}\sum_{i=1}^N w^T_i \log(A+I) w_i.  
\end{equation}
In the case of standard Gaussian random vectors, the variance of this estimator is $\frac{2}{N} \| \log(A + I) \|_F^2$. The quadratic forms $w^T_i \log(A+I) w_i$ can be approximated using $m$ steps of the Lanczos algorithm and the convergence is linked to the spectral distribution of the matrix $A + I$; see~\cite{GolubMeurant10} and Section~\ref{sec:lanczos} below.

To reduce the variance of~\eqref{eq:gh} and improve the convergence of the Krylov method for the approximation of the quadratic forms, one can introduce a symmetric positive definite (SPD) preconditioner $P \in \RR^{n \times n}$ and leverage the fact that
\begin{equation}\label{eq:preclog}
	\trace \log(A+I) = \trace \log P + \trace \log\left(P^{-\frac{1}{2}}(A+\ I)P^{-\frac{1}{2}}\right),
\end{equation}
which follows from the multiplicative property of the determinant. With a suitable choice of $P$, the term $\trace \log(P)$ can be computed cheaply and exactly, and the second term can be estimated via~\eqref{eq:gh}. 
This idea has been explored, for instance, in~\cite{FengKulickTang24,WengerPleissHennigCunninghamGardner22}.
A good preconditioner can be constructed as $P = S + I$, where $S$ is a low-rank approximation of $A$. In this sense,~\eqref{eq:lowrank} corresponds to the approximation of~\eqref{eq:preclog} in which the last term is completely ignored. A black-box preconditioner can be constructed by taking, as $S$, the Nystr\"om approximation~\cite{HalkoMartinssonTropp11,LiLindermanSzlamStantonKlugerTygert17}:
\begin{equation}\label{eq:nystrom}
	\hat{A}_{\ell}:=A\Omega (\Omega^T A \Omega)^\dagger\Omega^TA,
\end{equation}
for $\ell := k+p$, where $k$ is a target rank and $p$ is an oversampling parameter, $\Omega \in \RR^{n \times \ell}$ is a random Gaussian matrix with i.i.d.~$\mathcal N_{0,1}$ entries and $\dagger$ denotes the pseudoinverse of a matrix. It was shown in~\cite{FrangellaTroppUdell23} that, under suitable assumptions, letting $\hat P_\ell:=\hat A_\ell+I$, the preconditioned matrix $\hat M_\ell:=\hat P_\ell^{-\frac{1}{2}}(A+I)\hat P_\ell^{-\frac12}$ is well conditioned with high probability. This means that, when applying~\eqref{eq:gh} to the second term of~\eqref{eq:preclog}, both its variance and the number of Lanczos iterations are under control.

\subsection{Contributions} \label{sec:1s_strategy}

We propose and analyze a simple instance of~\eqref{eq:preclog}, where we use almost all the computational budget for the computation of a Nystr\"om preconditioner and we use a \emph{single} Gaussian random vector for estimating the log-determinant of the preconditioned matrix $\hat M_\ell \in \RR^{n \times n}$, i.e.
\begin{equation}\label{eq:residual}
\trace \log \left(\hat M_\ell \right) \approx w^T \log\left(\hat M_\ell\right)w .
\end{equation}
We show -- in practice and in theory -- that this strategy is effective when the matrix $A$ has at least moderate spectral decay. 
The advantage of our one-sample strategy over all possible instances of~\eqref{eq:preclog} is twofold: having a good preconditioner reduces the variance of the Girard-Hutchinson trace estimator, and the condition number of the matrix $\hat M_{\ell}$ is small so Lanczos quadrature is expected to be precise with only a small number of steps, see Figure~\ref{fig:Lanczos}. 

The only cases in which the one-sample strategy is not competitive is when $A$ has very slow singular value decay. To address the issue, we develop an adaptive strategy (log-det-ective, Algorithm~\ref{alg:logdetective}) that, in case low quality of the low-rank approximation is detected, changes the budget allocation to allow for more samples for the Girard-Hutchinson trace estimator. The numerical examples in Section~\ref{subsec:methods} show the effectiveness of this algorithm.

\subsection{Relation with existing literature} 
Since the vanilla Girard-Hutchinson trace estimator exhibits slow convergence, a variety of variance reduction techniques that apply to any symmetric (or SPSD) matrix have been developed. Most notably, Hutch++~\cite{MeyerMuscoWoodruff21} combines the Girard-Hutchinson estimator with randomized low-rank approximation, and an adaptive version (A-Hutch++) was developed in~\cite{PerssonCortinovisKressner22}. In the context of numerically low-rank matrices, these techniques have been extended and proved to be effective for matrix functions, relying on the identity $\trace f(A) = \trace f(P) + \trace \left (f(A)-f(P)\right )$, for some low-rank approximation $P$ of $A$. For instance, the paper~\cite{PerssonKressner23} proposes funNys++ for the case of matrix monotone functions.
For the specific case of the logarithm $f(x) = \log(1+x)$, these methods still need to apply the Lanczos quadrature to the original matrix $A$.
This means that they generally compare unfavorably with methods that leverage~\eqref{eq:preclog} to apply Lanczos quadrature to the preconditioned matrix instead, as we illustrate in Figure~\ref{fig:comparison_methods}.

The idea of ``preconditioning'' the problem of computing traces of matrix functions has been used, implicitly, in the Krylov-aware algorithm~\cite{ChenHallman23}; see also~\cite{ChenHuberLinZaid26}. Other methods, tailored to log-determinant computations, leverage the explicit preconditioner~\eqref{eq:preclog}: for instance, this was used in~\cite{SaibabaAlexanderianIpsen17} and~\cite{PerssonKressner23} to justify ignoring the $\mathrm{trace}\log \left ( P^{-1/2}(A+I)P^{-1/2}\right )$ and concentrating on finding a good matrix $P$. 
Other works combine the benefit of both variance reduction and preconditioning in this context: in~\cite{GardnerPleissWeinbergerBindelWilson18} the incomplete Cholesky preconditioner is used, in~\cite{WengerPleissHennigCunninghamGardner22} the authors analyzed the effectiveness of various preconditioners, and they proposed, in the numerical experiments, to split the budget of matvecs evenly between the computation of the preconditioner and SLQ; despite being very general, their analysis involves error bounds which depend on the size of the matrix $A$. In~\cite{FengKulickTang24}, the authors choose $P$ as the Nyström preconditioner of $A$, but no choice of parameters for the budget allocation and no convergence analysis is proposed. In our case, the analysis we propose for the Nystr\"om preconditioner is independent from the size of $A$.
%\textcolor{magenta}{However, despite being very general, the analysis of~\cite{WengerPleissHennigCunninghamGardner22} involves the size of the matrix $A$, which weakens the results when the problem grows. Conversely, in the specialized case of the Nyström preconditioner, the analysis we propose is independent from the size of $A$.}

Lastly, the idea of using a \emph{single} sample for trace estimation problems involving the Girard-Hutchinson estimator has appeared before and has been proven to be effective for a number of problems. More specifically, it has been used for trace estimation of matrices which are available via matrix-vector multiplications, XTrace~\cite{EpperlyTroppWebber24}, in the context of randomized probing methods~\cite{FrommerRinelliSchweitzer25}, and for spectral gap estimation problems~\cite{BenziRinelliSimunec24}. The idea of using only one sample for (preconditioned) log-determinant computations, to the best of our knowledge, is new.

\subsection{Outline}
In Section~\ref{sec:idealized} we analyze the proposed one-sample strategy for a fixed budget of total matvecs, considering an ideal preconditioner and quadratic forms involving the logarithm exactly computed. We then discuss and compare the obtained error with the low-rank strategy~\eqref{eq:lowrank} and instances of~\eqref{eq:preclog} where we split  the number of total matvecs among the creation of a preconditioner and the estimation of the trace of the preconditioned matrix according to some parameter. In Section~\ref{sec:nystrom} we derive an upper bound for the error produced by the one-sample strategy when using the Nyström preconditioner, assuming exact quadratic forms.
In Section~\ref{sec:detective_strategy} we propose Algorithm~\ref{alg:logdetective}, named \say{{log-det-ective}}, which re-allocates some budget of matvecs whenever it detects that the preconditioner is not good enough. %the one-sample strategy is not close to be the optimal strategy. 
Section~\ref{sec:experiments} contains numerical experiments on several matrices with various spectral decays: we plot the bounds obtained in Section~\ref{sec:nystrom} and we compare them with the bounds available for strategy~\eqref{eq:lowrank}, we show the behavior of the error produced by the approximation of the quadratic form and the behavior of the log-det-ective algorithm for different parameters, and we compare our proposed method with other existing strategies.

\section{Analysis of the proposed strategy in the idealized case} \label{sec:idealized}

We begin our analysis under simplified assumptions: the capability of computing exact quadratic forms with the logarithm (with a cost of $m$ matvecs) and the knowledge of an ideal preconditioner $P$ obtained from the best rank-$\ell$ approximation of $A$ (with a cost of $\ell$ matvecs). More specifically, let  $A=U \Lambda U^T$ be a spectral decomposition of $A$, where $U \in \mathbb{R}^{n \times n}$ is orthogonal and $\Lambda = \mathrm{diag}(\lambda_1, \ldots, \lambda_n)$ with $\lambda_1 \ge \ldots \ge \lambda_n \ge 0$ is the matrix containing the eigenvalues of $A$ in decreasing order. We consider the following preconditioner:
\begin{equation}\label{eq:idealP}
P_\ell := U_{\ell}\Lambda_\ell U_{\ell}^T+I = A_{\ell} + I,
\end{equation}
where $U_\ell \in \RR^{n \times \ell}$ is made of the first $\ell$ columns of $U$, $\Lambda_\ell \in \RR^{\ell \times \ell}$ is the leading $\ell \times \ell$ submatrix of $\Lambda$, and $A_\ell$ is the best rank-$\ell$ approximation of $A$ (in any unitarily invariant norm) by the Eckart-Young theorem. In this case, $\det P_\ell = \det(\Lambda_\ell + I)= \prod_{i=1}^\ell(1+\lambda_i)$ is cheap to compute, and the preconditioned matrix~\eqref{eq:residual} is
\begin{equation}\label{eq:Mk}
M_\ell = P_\ell^{-\frac12} (A+I) P_\ell^{-\frac12} = U \begin{bmatrix} I_\ell & \\ & \bar\Lambda_\ell + I_{n-\ell} \end{bmatrix} U^T,
\end{equation}
where $\bar\Lambda_\ell \in \RR^{(n-\ell)\times(n-\ell)}$ is the diagonal matrix containing the trailing $n-\ell$ eigenvalues of $A$.
We now take one sample for the Girard-Hutchinson trace estimator applied to $M_\ell$ and approximate
\begin{equation}\label{eq:onesample-idealized}
\tag{idealized one-sample}
\trace \log(A+I) \approx \trace \log\left( P_\ell \right) + w^T \log(M_\ell) w,
\end{equation}
where $w$ is a Gaussian random vector. We denote the error of this ``idealized'' one-sample approximation as 
\begin{align} \label{eq:err1S_ideal}
\mathrm{err_{\mathrm{1S_{id}}}}(A,\ell,w) := 
\left \lvert \trace\log(A+I) - \trace\log(P_\ell) - w^T(\log(M_\ell))w \right \rvert 
     = \left \lvert \trace \log(M_\ell) - w^T(\log(M_\ell))w \right \rvert.  
\end{align} 
Therefore, we have
\begin{equation} \label{eq:experrideal-one}
\mathbb{E}\left [\mathrm{err}_{\mathrm{1S_{id}}}^2(A,\ell,w)\right ] = 2 \|\log(M_\ell)\|_F^2 = 2\sum_{i=\ell+1}^n \log^2(1+\lambda_i).
\end{equation}
Conversely, if we simply apply the approximation~\eqref{eq:lowrank}, this time with a rank-$(\ell+m)$ preconditioner (to take into account the fact that we have extra budget because we do not compute the quadratic form), the error we obtain is
\begin{equation} \label{eq:errTrunc_ideal}
    \mathrm{err}_{\mathrm{LR_{id}}}(A,\ell+m) := \trace \log(A+I) - \trace \log (P_{\ell+m}) = \left \lVert \log(M_{\ell+m})\right \rVert_* = \sum_{i=\ell+m+1}^n \log(1+\lambda_i),
\end{equation}
where $\| \cdot \|_*$ denotes the nuclear norm, and $\mathrm{err}_{\mathrm{LR_{id}}}(A,\ell+m) \ge 0$ because $A+I \succeq P_\ell$. Informally, comparing~\eqref{eq:experrideal-one} and~\eqref{eq:errTrunc_ideal} we see that using one sample for the Girard-Hutchinson trace estimator improves the accuracy of the approximation we perform, making the error proportional to the Frobenius norm of the logarithm of the preconditioned matrix $M_\ell$, instead of its nuclear norm. This is advantageous unless $m$ is large and the matrix $A$ has a very fast spectral decay.

Moreover, we can consider a ``mixed'' strategy that uses the best rank-$\lfloor\alpha \ell\rfloor$ approximation of the matrix for some $\alpha \in [0,1]$ and $\left \lfloor\frac{\ell+m -\lfloor\alpha \ell \rfloor}{m}\right \rfloor$ samples for the Girard-Hutchinson estimator (for simplicity, from now on we assume that both $\alpha\ell$ and $\frac{\ell+m -{\alpha \ell}}{m}$ are integers), which is
\begin{equation} \label{eq:mixed_strategy} \tag{\text{idealized $\alpha$-rank}} 
\trace \log(A+I) \approx \trace\log\left(P_{\alpha \ell}\right) + \frac{m}{(1-\alpha)\ell+m} \sum_{i=1}^{\frac{(1-\alpha)\ell+m}{m}} w_i^T \log(M_{\alpha \ell}) w_i.
\end{equation}
Note that, for small values of $\alpha$, the method is close to strategy~\eqref{eq:gh}, while for large values of $\alpha$ it is close to~\eqref{eq:onesample-idealized}. 
This produces the error
\begin{equation*}
    \mathrm{err}_{\mathrm{\alpha R_{id}}}(A,\ell,m,\alpha,w_I) := \left \lvert \trace \log(A+I) - \trace\log\left(P_{\alpha \ell}\right) - \frac{m}{(1-\alpha)\ell+m} \sum_{i=1}^{\frac{(1-\alpha)\ell+m}{m}} w_i^T \log(M_{\alpha \ell}) w_i \right \rvert,
\end{equation*}
where we use $w_I$ as a shorthand for the set of random vectors $w_1, \ldots, w_{\frac{(1-\alpha)\ell+m}{m}}$.
In expectation, we have
\begin{equation} \label{eq:errMix_ideal}
    \mathbb{E}\left [\mathrm{err}_{\mathrm{\alpha R_{id}}}^2(A,\ell,m,\alpha,w_I)\right ] = \frac{2m}{(1-\alpha)\ell+m} \left \lVert \log(M_{\alpha \ell})\right \rVert_F^2.
\end{equation}
Therefore, for a fixed value of $\alpha$, using the mixed estimator~\eqref{eq:mixed_strategy} improves over the one-sample estimator~\eqref{eq:onesample-idealized} whenever 
\begin{equation} \label{eq:1S_vs_mix_ideal}
\frac{m}{(1-\alpha)\ell+m} \left \lVert \log(M_{\alpha \ell})\right \rVert _F^2 \le \left \lVert \log(M_{\ell})\right \rVert _F^2.
\end{equation}
Informally, the condition~\eqref{eq:1S_vs_mix_ideal} is satisfied when $A$ has slow spectral decay or when $\ell$ highly overestimates the numerical rank of the matrix $A$.

Figure~\ref{fig:optprec} shows a comparison of the average errors~\eqref{eq:experrideal-one},~\eqref{eq:errTrunc_ideal}, and~\eqref{eq:errMix_ideal} for six matrices, described in Section~\ref{subsec:matrices}, for a budget of matvecs ranging from $110$ to $1010$ and $m = 10$. Specifically, the plots show the square root of the error~\eqref{eq:experrideal-one} in %fat
thick blue lines, the error~\eqref{eq:errTrunc_ideal} in thick orange lines, and the square root of the error~\eqref{eq:errMix_ideal} computed for values of $\alpha$ ranging from $0$ {(which gives the approximation~\eqref{eq:gh})} to $0.9$ in thin lines ranging from purple (0) to blue (0.9). All expected errors have been normalized by the target quantity $\trace\log(A+I)$.

\begin{figure}[ht]
    \centering    \includegraphics[width=\textwidth]{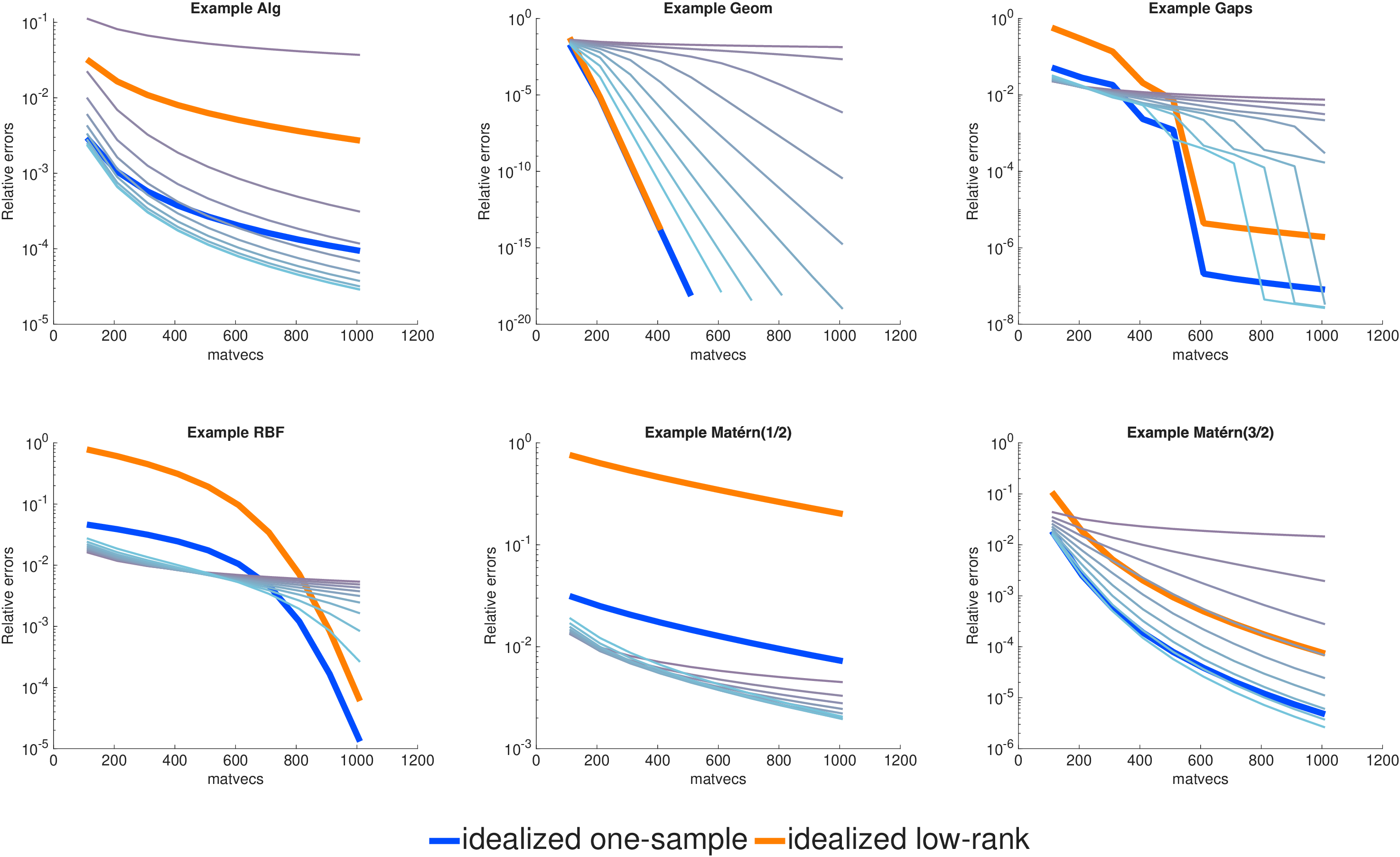}
    \caption{Square root of the variance~\eqref{eq:experrideal-one} for the~\eqref{eq:onesample-idealized} strategy divided by $\trace\log(A+I)$, error~\eqref{eq:errTrunc_ideal} for the~\eqref{eq:lowrank} strategy divided by $\trace\log(A+I)$ and square root of the variance~\eqref{eq:errMix_ideal} divided by $\trace\log(A+I)$ for~\eqref{eq:mixed_strategy} strategies for $\alpha=0,\ldots,0.9$, represented by thin lines, shading from purple ($\alpha=0$) to blue ($\alpha=0.9$) considering $m=10$ and the matrices $A\in \RR^{4000\times 4000}$ described in Section~\ref{subsec:matrices}.}
    \label{fig:optprec}
\end{figure}

\section{Using the Nystr\"om preconditioner}
\label{sec:nystrom}
Recalling the Nystr\"om approximation~\eqref{eq:nystrom}, our chosen preconditioner is
\begin{equation*}
\hat{P}_{\ell} := \hat A_{\ell} + I = A\Omega (\Omega^T A \Omega)^\dagger\Omega^TA + I,
\end{equation*}
where $\Omega \in \RR^{n\times\ell}$ is a Gaussian random matrix.
We call $\hat M_{\ell}$ the preconditioned matrix obtained from the Nyström preconditioner, namely
\begin{equation}\label{eq:NystromMkp}
\hat{M}_{\ell} := \hat P_{\ell}^{-\frac12} (A+I) \hat P_{\ell}^{-\frac12}.
\end{equation}
The approximation we consider becomes
\begin{equation}\label{eq:onesample-nystrom} \tag{one-sample}
\trace\log(A+I) \approx \trace\log\left(\hat P_\ell\right) + w^T \log\left(\hat M_\ell\right) w,
\end{equation}
where $w$ is a Gaussian random vector.
Let us also define the mixed strategy in the context of Nyström preconditioner, that is
\begin{equation} \label{eq:mixed_strategy-nystrom} \tag{\text{$\alpha$-rank}} 
\trace \log(A+I) \approx \trace\log\left(\hat P_{\alpha \ell}\right) + \frac{m}{(1-\alpha)\ell+m} \sum_{i=1}^{\frac{(1-\alpha)\ell+m}{m}} w_i^T \log(\hat M_{\alpha \ell}) w_i.
\end{equation}
In Section~\ref{subsec:analysis1sample} we derive a bound for the error attained by the approximation~\eqref{eq:onesample-nystrom} with the Nystr\"om preconditioner, and in Section~\ref{subsec:analysisOther} we do the same for~\eqref{eq:mixed_strategy-nystrom} and~\eqref{eq:lowrank}; in this section, we again assume that the quadratic forms with the logarithm are computed exactly using $m$ matvecs. Therefore, the total budget for evaluating~\eqref{eq:onesample-nystrom} is $\ell+m$ matvecs with $A$. 
We denote the error of~\eqref{eq:onesample-nystrom} as 
\begin{align}  \label{eq:err1S_Nystrom}
\mathrm{err_{\mathrm{1S_{Nys}}}}(A,\ell,w) :=  \left \lvert \trace \log\left(\hat M_{\ell}\right) - w^T\log\left(\hat M_{\ell}\right)w \right \rvert.
\end{align}

\subsection{Analysis of the one-sample strategy}\label{subsec:analysis1sample}
Using the law of total expectation, it follows from~\eqref{eq:err1S_Nystrom} that
\begin{align} \label{eq:experr1s_Nystrom}
\mathbb{E}\left[\mathrm{err}_{\mathrm{1S_{Nys}}}(A,\ell,w) \right]^2 \nonumber &= \mathbb{E}_\Omega \left[ \mathbb{E}_w\left[\mathrm{err}_{\mathrm{1S_{Nys}}}(A,\ell,w) \vert \Omega \right] \right]^2 \le \mathbb{E}_{\Omega}\left[ \left( \mathbb{E}_w\left[\mathrm{err}^2_{\mathrm{1S_{Nys}}}(A,\ell,w) \vert \Omega \right]\right)^{\frac{1}{2}} \right]^2 \nonumber 
\\ &= \mathbb{E}_{\Omega} \left[\left(2\left \lVert \log(\hat M_{\ell})\right \rVert^2_F\right)^{\frac{1}{2}} \right]^2
 = 2 \mathbb{E}_{\Omega} \left[\left \lVert \log(\hat M_{\ell})\right \rVert_F\right]^2,
\end{align}
where the inequality follows by the Cauchy-Schwarz inequality, and the second equality is the variance of the Girard-Hutchinson estimator. We use the expressions $\mathbb{E}_\Omega$ and $\mathbb{E}_w$ to denote the fact that the expectation is taken with respect to the random variable $\Omega$ or $w$, respectively. To derive a bound for the right-hand side of~\eqref{eq:experr1s_Nystrom}, we need to recall some general results on the  Nyström approximation. Since $A \succeq \hat A_{\ell} \succeq 0$ (see, e.g.,~\cite[Lemma 1]{GittensMahoney16}), we have that $\hat{P}_{\ell} = \hat{A}_{\ell} + I \succeq I$. Moreover, $x \mapsto \log(1+x)$ is an operator monotone function, that is, a function that preserves the Löwner order of SPSD matrices, see~\cite[Example 3.6]{Chansangiam13}; we thus have 
% A crucial property of the approximation $\hat{A}_{\ell}$ is given in \cite[Lemma 1]{GittensMahoney16} and is the following:
% \begin{equation} \label{eq:nysorder}
% A \succeq \hat A_{\ell} \succeq 0,
% \end{equation}
% which leverages the fact that $\hat{A}_\ell=A^{\frac{1}{2}}\Pi_{A^{\frac{1}{2}}\Omega}A^{\frac{1}{2}}$.
% Observe also that $\hat{A}_{\ell}\succeq0$ is equivalent to $\hat{P}_{\ell} = \hat{A}_{\ell} + I \succeq I.$
% The same ordering can be transported into the function $\log(1+x)$, exploiting the fact that it is an operator monotone function, i.e. a function $f$ for which 
% \[A\succeq B \quad \text{ implies } \quad f(A)\succeq f(B),
% \] see~\cite[Example 3.6]{Chansangiam13}. We thus have:
\begin{equation} \label{eq:lognysorder}
    \log(A+I)\succeq \log(\hat A_{\ell}+I)  = \log \left ( \hat P_{\ell} \right )\succeq 0.
\end{equation}
Moreover, we can write %An alternative form of the preconditioned term $\hat M_{\ell}$ exploits the following identity:
\begin{equation} \label{eq:equivalentMkp}
     \hat M_{\ell} = \hat P^{-\frac{1}{2}}_{\ell} \left(A+I\right) \hat P^{-\frac{1}{2}}_{\ell} = \hat P_{\ell}^{-\frac12} \left ( A - \left (\hat P_{\ell} - I \right) \right ) \hat P_{\ell}^{-\frac12} + I = \hat P_{\ell}^{-\frac12} \left ( A - \hat A_{\ell} \right ) \hat P_{\ell}^{-\frac12} + I,
\end{equation}
which implies $\hat M_\ell \succeq I$ and 
\begin{equation}\label{eq:logMl}
\log\left(\hat M_\ell\right) = \log\left(\hat P_{\ell}^{-\frac12} \left ( A - \hat A_{\ell} \right ) \hat P_{\ell}^{-\frac12}+I\right)\succeq 0.
\end{equation}
We here give a bound of the expectation of the Frobenius norm of the matrix $\log(\hat M_{\ell})$, taking advantage of the properties above.
\begin{lemma} \label{lemma:NysMk_fro}
For any SPSD matrix $A \in \mathbb{R}^{n \times n}$, for any $k\ge0$ and $p \ge 2$ such that $k+p=\ell$, with $\ell \ge 4$, we have
\begin{equation} 
\mathbb{E} \left[\left \lVert \log\left(\hat M_{\ell}\right)\right \rVert_F\right]^2 \le \left( 1+ \frac{k}{p-1} \right) \log \left( 1+ \frac{2e^2(k+p)}{p^2-1}\left \lVert \bar\Lambda_k\right \rVert_* + \left(1+\frac{2k}{p-1}\right) \left\lVert \bar\Lambda_k \right \rVert_2 \right)  \left \lVert \log\left(\bar\Lambda_k+I\right) \right \rVert_*,   
\end{equation}
where $\bar \Lambda_k \in \RR^{(n-k)\times (n-k)}$ is the diagonal matrix that contains the trailing $n-k$ eigenvalues of $A$. 
\end{lemma}
\begin{proof}
We have
\begin{align}
\mathbb{E} \left[\left \lVert \log\left(\hat M_{\ell}\right)\right \rVert_F\right]^2 
& \overset{}{\le}  \mathbb{E}\left [ \left \lVert \log\left(\hat M_{\ell}\right)\right \rVert^\frac{1}{2}_* \left \lVert \log\left(\hat M_{\ell}\right)\right \rVert^\frac{1}{2}_2\right ]^2 \overset{}{\le} \mathbb{E} \left [ \left \lVert \log\left(\hat M_{\ell}\right)\right \rVert_* \right ]  \mathbb{E} \left [ \left \lVert \log\left(\hat M_{\ell}\right)\right \rVert_2\right],
 \label{eq:proof1}
\end{align} 
where the first inequality follows from the fact that we have $\|H\|_F \le \|H\|^{\frac{1}{2}}_* \|H\|^{\frac{1}{2}}_2$ for any matrix $H$ and the second 
inequality follows from Cauchy-Schwarz. Let us bound the two terms appearing in the right-hand-side of~\eqref{eq:proof1} separately. We have
\begin{align*}
\left \lVert \log\left(\hat M_{\ell}\right)\right \rVert_* &\mathrel{\overset{\eqref{eq:logMl}}{=}} \trace \log\left(\hat P^{-\frac{1}{2}}_{\ell} (A+I)\hat P^{-\frac{1}{2}}_{\ell}\right) = \trace \log (A+I)-\trace \log\left(\hat P_{\ell}\right) \\& \hspace{1.525mm}= \trace \left(\log (A+I)- \log\left(\hat P_{\ell}\right)\right)  \overset{\eqref{eq:lognysorder}}{=} \left \lVert\log (A+I)- \log\left(\hat P_{\ell}\right)\right \rVert_*. \end{align*}
Using~\cite[Theorem 3.9]{PerssonKressner23} we have
\begin{equation}\label{eq:nuclear}
\mathbb{E} \left [\left \lVert \log (A+I)- \log\left(\hat P_{\ell}\right)\right \rVert_* \right]
\le 
\left(1+\frac{k}{p-1}\right)\left \lVert \log\left(\bar\Lambda_{k}+I\right)\right \rVert_*.  
\end{equation}
Let us now bound the second term of~\eqref{eq:proof1}: we have
\begin{align*}
\left \lVert \log\left(\hat M_{\ell}\right)\right \rVert_2 & \overset{\eqref{eq:equivalentMkp}}{=}  \left \lVert \log\left( \hat P^{-\frac{1}{2}}_{\ell}\left(A-\hat A_{\ell}\right) \hat P^{-\frac{1}{2}}_{\ell} + I \right) \right \rVert_2  \overset{}{=} \log\left(1+\left \lVert \hat P^{-\frac{1}{2}}_{\ell}\left(A-\hat A_{\ell}\right)\hat P^{-\frac{1}{2}}_{\ell} \right \rVert_2  \right) \\& \hspace{1.525mm}
\le \log\left (1+ \left \lVert\hat P^{-\frac{1}{2}}_{\ell} \right \rVert^2_2 \left \lVert A-\hat A_{\ell}\right \rVert_2 \right ) 
\overset{}{\le} \log\left (1+ \left \lVert A-\hat A_{\ell}\right \rVert_2 \right ), 
\end{align*}
where the last equality follows from the fact that $\hat P_\ell \succeq I$.
% Where equality $i.$ follows from~\eqref{eq:equivalentMkp}, equality $ii.$ leverages the fact that $\log\left( \hat P^{-\frac{1}{2}}_{\ell}\left(A-\hat A_{\ell}\right) \hat P^{-\frac{1}{2}}_{\ell} + I \right) \succeq 0$, and that for any SPD matrix $H = U \Sigma U^T$ we have  \[\left \lVert \log(U \Sigma U^T+I) \right \rVert_2 = \left \lVert U \log( \Sigma +I) U^T \right \rVert_2 = \left \lVert \log( \Sigma +I) \right \rVert_2 = 
% \log(\sigma_1+1),\] whereas inequality $iii.$ follows from the fact that $\left \lVert\hat P^{-\frac{1}{2}}_{\ell} \right \rVert^2_2 = \left \lVert\hat P^{-1}_{\ell} \right \rVert_2 \le 1$, since $P^{-1}_{\ell} \preceq I$, and from the monotonicity of the function $\log(1+x)$.
Therefore,
\begin{align}
    \mathbb{E}\left [\left \lVert \log\left(\hat M_{\ell}\right)\right \rVert_2 \right ] & \le \mathbb{E} \left [ \log\left(1+\left \lVert A-\hat A_{\ell}\right \rVert_2\right)\right]  \le \log\left (1+ \mathbb{E}\left[\left \lVert A-\hat A_{\ell}\right \rVert_2 \right]\right )\nonumber \\
    & \le \log \left ( 1 + \frac{2e^2(k+p)}{p^2-1} \left \lVert \bar \Lambda_k\right \rVert_* + \left(1+\frac{2k}{p-1}\right) \left \lVert \bar \Lambda_k\right \rVert_2  \right ),\label{eq:spectral}
\end{align}
where the second inequality follows from Jensen applied to the concave function $x \mapsto \log(1+x)$ and the last inequality is~\cite[Proposition 2.2]{FrangellaTroppUdell23}. Plugging~\eqref{eq:nuclear} and~\eqref{eq:spectral} into~\eqref{eq:proof1} concludes the proof.  
\end{proof}

Lemma~\ref{lemma:NysMk_fro} together with~\eqref{eq:experr1s_Nystrom} immediately implies the following bound for the~\eqref{eq:onesample-nystrom} strategy.
\begin{theorem} \label{th:err_practical}
Let $A \in \mathbb{R}^{n \times n}$ be a SPSD matrix, let $w \sim \mathcal{N}_{0,I}$, let $k\ge0$ and $p \ge 2$ such that $k+p=\ell$, with $\ell \ge 4$. Then
\begin{equation} \label{eq:exp_logdet_true_fro}\mathbb{E} \left [\mathrm{err}_{\mathrm{1S_{Nys}}}(A,\ell,w)\right]^2
\le 2 \left( 1+ \frac{k}{p-1} \right) \log\left( 1+ \frac{2e^2(k+p)}{p^2-1}\left \lVert \bar\Lambda_k \right \rVert_* + \left(1+\frac{2k}{p-1}\right)\left \lVert \bar\Lambda_k \right \rVert_2 \right)  \left \lVert \log(\bar\Lambda_k+I) \right \rVert_*,
\end{equation} 
where $\bar \Lambda_k \in \RR^{(n-k)\times (n-k)}$ is the diagonal matrix that contains the trailing $n-k$ eigenvalues of $A$. 
\end{theorem}

\subsection{Comparison with low-rank and \texorpdfstring{$\alpha$}{alpha}-rank}\label{subsec:analysisOther}

We can re-do the same computation in the case of the~\eqref{eq:mixed_strategy-nystrom} strategy of parameter $\alpha$, that uses $\alpha\ell$ matvecs for the Nyström approximation and the remaining $\frac{(1-\alpha)\ell + m}{m}$ matvecs for the Girard-Hutchinson estimator. This choice of parameters is to enforce a total budget of matvecs of $\ell+m$, in order to make a fair comparison with our one-sample strategy. 
The error of the~\eqref{eq:mixed_strategy-nystrom} strategy is given by
\begin{align*} \label{eq:errmix_Nystrom}
\mathrm{err}_{\mathrm{\alpha R_{Nys}}}(A,\ell,m,\alpha,w_I) :&= %\left \lvert \trace \log(A+I) - \trace\log(\hat P_{\lfloor\alpha \ell\rfloor}) - \frac{1}{\ell+1-\lfloor\alpha \ell\rfloor} \sum_{i=1}^{\ell+1-\lfloor\alpha \ell\rfloor} w_i^T (\log(\hat M_{\lfloor\alpha \ell\rfloor})) w_i \right \rvert  \\&=
\left \lvert \trace\log\left(\hat M_{\alpha \ell}\right) - \frac{m}{(1-\alpha)\ell+m} \sum_{i=1}^{\frac{(1-\alpha)\ell+m}{m}} w_i^T \left(\log\left(\hat M_{\alpha \ell}\right)\right) w_i \right \rvert. 
\end{align*}
Following the same steps as in~\eqref{eq:experr1s_Nystrom}, we get
\begin{align} \label{eq:experrmix_Nystrom}
\mathbb{E}\left[\mathrm{err}_{\mathrm{\alpha R_{Nys}}}(A,\ell,m,\alpha,w_I) \right]^2 \nonumber %&= \mathbb{E}_\Omega \left[ \mathbb{E}_w\left[\mathrm{err}_{\mathrm{Mix_{Nys}}}(A,\ell,\alpha,w) \vert \Omega \right] \right]^2 \nonumber
%\\&\le \mathbb{E}_{\Omega}\left[ \left( \mathbb{E}_w\left[\mathrm{err}^2_{\mathrm{Mix_{Nys}}}(A,\ell,\alpha,w) \vert \Omega \right]\right)^{\frac{1}{2}} \right]^2 \nonumber
%\\ &= \mathbb{E}_{\Omega} \left[\left(\frac{2}{\ell+1-\lfloor\alpha \ell\rfloor}\|\log(\hat M_{\lfloor\alpha\ell\rfloor})\|^2_F\right)^{\frac{1}{2}} \right]^2 \nonumber
%\\ &
= \frac{2m}{(1-\alpha)\ell+m} \mathbb{E}_{\Omega} \left[\left \lVert \log\left(\hat M_{\alpha\ell}\right)\right \rVert_F\right]^2
\end{align}
and using Lemma~\ref{lemma:NysMk_fro} with $ \alpha\ell$ instead of $\ell$ we get, for any $k \ge 0$, $p \ge 2$ with $k + p =  \alpha \ell $, {and $\alpha \ell\ge4$,}
\begin{multline}\label{eq:exp_logdet_true_mix} \mathbb{E} \left [\mathrm{err}_{\mathrm{\alpha R_{Nys}}}(A,\ell,m,\alpha,w_I)\right]^2  
\\  \le \frac{2m}{(1-\alpha)\ell+m} \left( 1+ \frac{k}{p-1} \right) \log\left( 1+ \frac{2e^2(k+p)}{p^2-1}\left \lVert \bar\Lambda_k\right \rVert_* + \left(1+\frac{2k}{p-1}\right) \left \lVert \bar\Lambda_k\right \rVert_2 \right)  \left \lVert \log\left(\bar\Lambda_k+I\right) \right \rVert_*.
\end{multline}   
Finally, consider the approximation~\eqref{eq:lowrank}, whose error is
\begin{equation} \label{eq:errTrunc_Nys}
\mathrm{err}_{\mathrm{LR_{Nys}}}(A,\ell+m):=\trace \log\left(A+I\right)-\trace\log\left(\hat A_{\ell+m} +I\right).
\end{equation} 
Equation~\eqref{eq:nuclear}, which was used in the analysis of the approximation~\eqref{eq:onesample-nystrom} and is precisely~\cite[Theorem 3.9]{PerssonKressner23}, states that for any $k\ge 0$, $p\ge 2$ such that $k+p=\ell+m$, we have that
\begin{equation} \label{eq:err:Trunc_Nys_leading}
\mathbb{E} \left [\mathrm{err}_{\mathrm{LR_{Nys}}}(A,\ell+m)\right]
\le 
\left(1+\frac{k}{p-1}\right)\left \lVert \log\left(\bar\Lambda_{k}+I\right)\right \rVert_*.
\end{equation}

The bound~\eqref{eq:err:Trunc_Nys_leading} is difficult to compare directly with the bounds~\eqref{eq:exp_logdet_true_fro} and~\eqref{eq:exp_logdet_true_mix}, one reason being that they are true for any combination of $k$ and $p$ with fixed sum (and the optimal values also depend on the chosen method). 
While~\eqref{eq:experrideal-one} predicts that the squared error of the ideal one-sample strategy is proportional to $\| \log(\bar \Lambda_k + I)\|_F^2$, {introducing the Nystr\"om approximation makes the quantity $\|\log(\bar \Lambda_k+I)\|_*$ appear in the bound for the one-sample strategy~\eqref{eq:exp_logdet_true_fro}.} This makes the benefit of the one-sample strategy over the low-rank strategy less clear from a theoretical point of view.
For these reasons, we postpone a more thorough comparison to the numerical experiments in Section~\ref{sec:experiments}, where we compare all three strategies on different examples, and we also compare the theoretical bounds for the one-sample strategy~\eqref{eq:exp_logdet_true_fro} and the low-rank strategy~\eqref{eq:err:Trunc_Nys_leading}. We will show, empirically, that the upper bound~\eqref{eq:exp_logdet_true_fro} is often much smaller than~\eqref{eq:err:Trunc_Nys_leading}.

\subsection{Inexact quadratic forms with Lanczos method} \label{sec:lanczos}
In practice, we approximate quadratic forms $w^T \log(B) w$, for an SPD matrix $B \in \mathbb{R}^{n \times n}$, with a Gaussian quadrature approach that uses $m$ matvecs~\cite{GolubMeurant10,Saad94}. We run $m$ iterations of the Lanczos algorithm with starting vector $w$, resulting in a tridiagonal matrix $T_m \in \mathbb{R}^{m \times m}$; we then take the approximation
\begin{equation}\label{eq:quadrature}
w^T \log(B) w \approx \|w\|^2 e_1^T \log(T_m) e_1,
\end{equation}
where $e_1 \in \mathbb{R}^m$ is the first vector of the canonical basis. The leading cost is the computation of $m$ matvecs with $B$. Note that, if $B = \hat M_\ell$ or $B = A + I$, this amounts to $m$ matvecs with $A$. The error of the approximation~\eqref{eq:quadrature} can be bounded using exactness properties of Gauss quadrature rules for polynomials combined with polynomial approximation results for the logarithm. In particular, it was shown in~\cite[Corollary 4]{CortinovisKressner21} that
\begin{equation}\label{eq:errorlanczos}
\text{Lanc}(B, m,w) := \left \lvert w^T\log(B)w-\|w\|^2 e_1^T\log(T_m)e_1\right \rvert\leq c_B \|w\|^2 \Bigg{(}\frac{\sqrt{\kappa(B)+1}-1}{\sqrt{\kappa(B)+1}+1}\Bigg{)}^{2m},
\end{equation}
where $c_B = 2\Big{(} \sqrt{\kappa(B)+1}+1 \Big{)} \log(2\kappa(B))$ and $\kappa(B)$ denotes the condition number of $B$ with respect to the spectral norm. 

If we were to use~\eqref{eq:errorlanczos} to analyze our one-sample strategy, we would run into the issue that, when $w \sim \mathcal{N}_{0,I}$, we typically have $\|w\| \approx \sqrt{ n}$. {Therefore,~\eqref{eq:errorlanczos} only predicts a small error for values of $m$ which are proportional to $\log(n)$, even in the presence of a well conditioned matrix $B$. However, in practice, we observed that a small constant value of $m$, e.g. $m=10$,} is sufficient to get an approximation error which is small with respect to the overall error of the one-sample strategy; see Section~\ref{sec:experiments} below.

\section{Error control and a ``detective'' version}
\label{sec:detective_strategy}
While the method~\eqref{eq:onesample-nystrom} is in general very effective in approximating the regularized log-determinant of an SPSD matrix $A$, there are situations in which its error~\eqref{eq:err1S_Nystrom} can be quite large, or suboptimal with respect to other strategies: when $A$ has a very slow singular value decay, or when $\ell$, the budget of matvecs allocated to the Nystr\"om approximation of $A$, is much larger than the numerical rank of $A$.

We propose to overcome the issue by checking, numerically, whether a Nystr\"om approximation of high rank is really advantageous over one of smaller rank. More specifically, assume that we have a budget of $\ell + m$ matvecs. For a fixed parameter $\beta \in (0,1)$, assuming for simplicity that $\beta \ell$ and $\beta^2 \ell$ are integers, we compare the rank-$ \beta \ell $ Nystr\"om approximation of $A$ with the one of rank-$\beta^2 \ell$. If we are satisfied with the improvement of the quality of the low-rank approximation from rank $\beta^2\ell$ to rank $\beta \ell$, we apply the one-sample strategy: we enlarge the Nystr\"om approximation using a total of $\ell$ matvecs and then use SLQ with one random Gaussian vector using $m$ Lanczos steps.
Otherwise, we apply the~\eqref{eq:mixed_strategy-nystrom} strategy with $\alpha=\beta$.

In the rest of this section, we detail how to check the Nystr\"om approximation error, we propose a heuristic method to check whether the one-sample strategy is appropriate, and we explain how to update the Nystr\"om approximation. 
The procedure is summarized in Algorithm~\ref{alg:logdetective}.

\begin{algorithm} \caption{log-det-ective} \label{alg:logdetective}
\begin{algorithmic}[1]
  \Statex \textbf{Input:} SPSD matrix $A\in\mathbb{R}^{n\times n}$; total budget of matvecs $\ell+m$; parameter $\beta \in (0,1)$.
  \Statex \textbf{Output:} Approximation $\trace_{(\ell,m)}\log(A+I) \approx \trace\log(A+I)$
  \vspace{0.5pc}
  \State Sample $\Omega \in \RR^{n \times \lfloor \beta\ell\rfloor}$ with i.i.d. $\mathcal{N}_{0,1}$ entries
  \State Build the Nyström approximation $\hat{A}_{\lfloor\beta\ell\rfloor}$ by~\cite[Algorithm 3]{TroppYurtsverUdellCevher17pt2}
  \State Compute $\mathrm{err}_F(A,\beta\ell)\approx \left \lVert A-\hat{A}_{\lfloor\beta\ell\rfloor}\right \rVert _F$ and $\mathrm{err}_F\left(A,\beta^2 \ell\right)\approx \left \lVert A-\hat{A}_{\lfloor\beta^2 \ell\rfloor}\right \rVert_F$ using~\eqref{eq:leave1out} 
  \If{ inequality~\eqref{eq:practical_threshold} is true, }\Comment{\emph{Use one-sample strategy}}
    \State Sample $\Psi \in \RR^{n \times (\ell-\lfloor \beta\ell\rfloor)}$ with i.i.d. $\mathcal{N}_{0,1}$ entries
    \State Update the Nystr\"om approximation, getting $\hat{A}_{\ell}$ \Comment{\emph{See subsection}~\ref{sec:nystrom_update}}
    \State Compute $t_1:=\trace\log\left(\hat{P}_\ell\right) = \trace\log\left(\hat A_\ell + I\right)$
    \State Use stochastic Lanczos quadrature to compute $t_2 := \text{Lanczos}(\log,A,\hat{P}_\ell,w,m)$ for $w \sim \mathcal N_{0,I}$
  \Else\Comment{\emph{Use strategy~\eqref{eq:mixed_strategy-nystrom}}}
    \State Compute $t_1:=\trace\log\left(\hat{P}_{\lfloor\beta\ell\rfloor}\right)$
    \State Set $N := \lfloor \frac{\ell+m-\lfloor \beta \ell \rfloor}{m} \rfloor$ and sample $N$ vectors $w_i \sim \mathcal{N}_{0,I}$
    \State Use stochastic Lanczos quadrature to compute $t_2:=\frac{1}{N}
      \sum_{i=1}^N\text{Lanczos}
      (\log,A,\hat P_{\lfloor\beta\ell\rfloor},w_i,m)$
  \EndIf
  \Statex \textbf{Return} \(
  \trace_{(\ell,m)}\log(A+I) := t_1 + t_2. \)
\end{algorithmic}
\end{algorithm}

\subsection{Evaluating the error of the Nystr\"om approximation}
We compute an approximation of the Frobenius norm of the error produced by the Nyström approximation using the leave-one-out estimator developed in~\cite{EpperlyTropp24}, which is
\begin{equation}\label{eq:leave1out}
\mathrm{err}^2_F\left(A,\beta\ell\right) := \frac{1}{\beta\ell}\sum_{i=1}^{\beta\ell}\left \lVert (A-\hat A_{\Omega^{(i)}})\omega_i\right \rVert^2 \approx \left \lVert A-\hat{A}_{\beta\ell}\right \rVert ^2_F,
\end{equation}
where $\Omega^{(i)} \in \mathbb{R}^{n \times ( \beta\ell -1)}$ is the matrix $\Omega$ with the $i$-th column removed, {$\hat{A}_{\Omega^{(i)}}$ is the Nyström approximation that results from leaving the
$i$-th column out of $\Omega$}, and $\omega_i$ is the $i$-th column of $\Omega$. Roughly speaking, this amounts to estimating the Frobenius norm of the $( \beta\ell -1)$-Nystr\"om approximation  by multiplication with one random vector, and averaging the estimate over all $ \beta\ell $ choices of the vector which is ``left out''. Once the factors of the rank-$ \beta \ell $ Nystr\"om approximation have been computed, the expression~\eqref{eq:leave1out} can be evaluated by an additional $\mathcal{O}(\beta^3\ell^3)$ operations; see~\cite[Algorithm 4.1]{EpperlyTropp24}.

\subsection{Choice of \texorpdfstring{$\beta$}{beta}}
The parameter $\beta$ needs to be chosen small enough so that, if we detect that the Nystr\"om approximation is not helpful enough in reducing the variance, we can switch to SLQ and we still have a substantial budget of matvecs for that part. At the same time, the parameter $\beta$ needs to be large enough for the $ \beta\ell $-Nystr\"om approximation to be comparable with the $\ell$-Nystr\"om approximation. 

Choosing $\beta = 3/4$ seems to provide a good trade-off between these two goals; see Section~\ref{sec:experiments}
below. Moreover, under the simplified assumption that quadratic forms are computed exactly with $m$ matvecs (and that quantities are integers), we have
\begin{equation*}
\mathbb{E}\left [\mathrm{err}_{\mathrm{\alpha R_{Nys}}}^2\left(A,\ell,\frac{3}{4},w\right)\right ] = \frac{8m}{\ell+4m} \mathbb{E}\left [ \left \lVert \log\left(\hat M_{\frac 34 \ell}\right)\right \rVert_F^2 \right ]\leq \frac{8m}{\ell+m} \|\log(A+I)\|_F^2;
\end{equation*}
this means that, in the case in which the Nystr\"om approximation is not good, the variance of the output of Algorithm~\ref{alg:logdetective} with $\beta = 3/4$ is at most $4$ times worse than plain stochastic Lanczos quadrature~\eqref{eq:gh} applied to $A$.

\subsection{Condition for switching the strategy}
We would like to compare, a priori, the one-sample strategy~\eqref{eq:onesample-nystrom} with budget $\ell+m$ and the mixed strategy~\eqref{eq:mixed_strategy-nystrom} with the same budget and $ \frac{ (1-\beta)\ell + m}{m}  $ vectors for the SLQ part. Even in the ideal case in which quadratic forms can be computed exactly with $m$ matvecs, testing condition~\eqref{eq:1S_vs_mix_ideal} would require computing quantities that we do not have. Therefore, as a proxy for these, we compare, instead, the one-sample strategy with a budget of $\beta \ell$ matvecs for the low-rank approximation and $m$ matvecs for the SLQ, and the mixed strategy~\eqref{eq:mixed_strategy-nystrom} with parameter $\beta$ and the same budget of matvecs. 
%(with the remaining $\left \lfloor \frac{ (m+\lfloor\alpha\ell\rfloor-\lfloor\alpha\ell/2\rfloor)}{m} \right \rfloor$ samples).
%The natural candidate to check whether the low-rank approximation is improving enough is given by the condition~\eqref{eq:1S_vs_mix_ideal} for rank $\lfloor\alpha \ell\rfloor$, $\alpha=\frac{1}{2}$, and $m+\lfloor\alpha \ell\rfloor$Precisely as the total budget of matvecs.
Assuming for simplicity that $\beta^2\ell$ is integer, condition~\eqref{eq:1S_vs_mix_ideal} can be rewritten as
\begin{equation} \label{eq:ideal_threshold}
\frac{m}{(1-\beta)\beta\ell+m} \left \lVert \log\left(M_{\beta^2\ell}\right)\right \rVert^2_F \ge \left \lVert \log(M_{\beta\ell})\right \rVert ^2_F.  
\end{equation}
Since checking this condition is too expensive, we consider instead the approximation
\begin{equation}\label{eq:manyapprox}
    \left \lVert \log \left(M_{\beta\ell}\right)\right\rVert_F^2 = \sum_{i=\beta\ell+1}^n \log(1+\lambda_i)^2 \approx \sum_{i=\beta\ell+1}^n \lambda_i^2 \approx \left \lVert A - \hat A_{\beta\ell}\right \rVert_F^2 \approx \mathrm{err}_F^2(A, \beta\ell),
\end{equation}
where the first approximation holds whenever the eigenvalues $\lambda_i$, for $i \ge \beta\ell+1$, are sufficiently small, the second approximation makes sense as $\hat A_{\beta\ell}$ is a rank-$\beta\ell$ approximation of $A$, and the third is the estimator~\eqref{eq:leave1out}; an analogous approximation can be computed for $\| \log (M_{\beta^2\ell})\|_F^2$. In our proposed algorithm, we check the condition
\begin{equation} \label{eq:practical_threshold} \frac{m}{(1-\beta)\beta\ell+m} \mathrm{err}^2_F\left(A,\beta^2\ell\right) \geq  \mathrm{err}^2_F\left(A,\beta\ell\right); 
\end{equation}
if it is satisfied, we choose the one-sample strategy.

\begin{remark}
    Checking the condition~\eqref{eq:practical_threshold} makes it so that Algorithm~\ref{alg:logdetective} is more likely to suggest the strategy~\eqref{eq:onesample-nystrom} over the strategy~\eqref{eq:mixed_strategy-nystrom}, with respect to the (idealized) case in which we check~\eqref{eq:ideal_threshold}, because of the inequality $\log(1+x) \le x$.
\end{remark}

\subsection{Updating the Nystr\"om approximation}
\label{sec:nystrom_update}

The initial rank-$\beta\ell$ Nystr\"om approximation is computed by~\cite[Algorithm 3]{TroppYurtseverUdellCevher17} by sketching $A$ with a Gaussian matrix $\Omega \in \RR^{n \times \beta\ell}$, with a computational cost of $\beta\ell$ matrix-vector multiplications with $A$ plus $\mathcal O(n\beta^2\ell^2)$. If condition~\eqref{eq:practical_threshold} holds, and therefore we need to compute the rank-$\ell$ Nyström approximation, we sample another Gaussian matrix $\Psi\in \RR^{n \times (\ell-\beta\ell)}$ and we enlarge the sketch to compute the new Nystr\"om approximation. Note that the sketch $A\Omega$ can be reused; therefore, the additional cost consists in $\ell-\beta\ell$ matvecs with $A$ plus $\mathcal O(n \ell^2)$ operations.

\section{Numerical Experiments} \label{sec:experiments}

In this section, we show the behavior of our proposed methods on a variety of examples. All  experiments are implemented in 
MATLAB R2025b. Scripts to reproduce all figures in this paper are available at \url{https://github.com/dantoni2000/log-det-ective}. Section~\ref{subsec:matrices} describes the matrices that we use for most of the experiments; Section~\ref{subsec:bounds} contains a comparison of the actual behavior of the one-sample and low-rank strategy with the corresponding bounds; Section~\ref{subsec:expsLanczos} analyzes the convergence of the Lanczos method for computing quadratic forms; Section~\ref{subsec:expsbetas} addresses different choices of $\beta$ in Algorithm~\ref{alg:logdetective}. Section~\ref{subsec:methods} compares Algorithm~\ref{alg:logdetective} and the one-sample strategy~\eqref{eq:onesample-nystrom} with other existing algorithms for approximating the log-determinant, and Section~\ref{subsec:large_examples} does the same for kernel matrices of increasing size.

\subsection{Example matrices}\label{subsec:matrices}
Usually, in applications one is interested in computing a \emph{regularized} log-determinant $\log\det(H + \mu I)$, for a suitable parameter $\mu > 0$. This problem can be reduced to our setting~\eqref{eq:logdet} by the relation
\begin{equation*}
    \log(H + \mu I ) = \log\mu \cdot I + \log(A + I), \qquad A := \frac{1}{\mu} H.
\end{equation*}

%\begin{remark}We remark that using a Rademacher vector $w$ instead of a standard Gaussian random vector would recover the ``diagonal term'' $\trace \log\mu \cdot I$ exactly without the need to scale the problem. We chose to stick to Gaussian random vectors in our numerical experiments, as we did not see a great benefit in using Rademacher vectors we recommend to use them in practice; however, the analysis in this case involves the matrix $\log(B)-\diag\log(B)$, for $B = \hat M_\ell$ or $B = A+I$. This quantity has less interpretation, as we have few information over the matrix $\diag\log(B)$.
% \end{remark}

In this section, we will describe how to construct the matrix $H$, what value we choose for $\mu$, and then we will run all algorithms on the corresponding matrix $A$. Unless otherwise indicated, our matrices have size $n = 4000$.

\paragraph{Synthetic Matrices} \label{subsec:syntm}
The first matrix is inspired by~\cite[Equation (19)]{PerssonKressner23} and has algebraic spectral decay: we set
\begin{equation} \label{eq:example1}
\tag{Example Alg}
H_{\text{alg}}:=U \Lambda_{\text{alg}} U^T, \quad (\Lambda_{\text{alg}})_{i,i}=i^{-2},   
\end{equation}
for some orthonormal matrix $U \in \mathbb{R}^{n \times n}$, and we set $\mu = 10^{-2}$. We expect that this toy problem will be simple to tackle, in fact the matrix $A+I$ is well conditioned and it has a consistent spectral decay. 
The second synthetic matrix has a geometric spectral decay:
\begin{equation} \label{eq:example2}
\tag{Example Geom}
H_{\text{geom}}:=U \Lambda_{\text{geom}} U^T, \quad (\Lambda_{\text{geom}})_{i,i}=e^{-0.1 i
}, \quad \mu = 10^{-4}.
\end{equation}
The last synthetic matrix is built in such a way that we have control over all the gaps and all the clusters of the eigenvalues of the matrix $H$. It is inspired by~\cite{SorensenEmbree16}, generalized with several gaps as
\begin{equation} \label{eq:example3}
\tag{Example Gaps}
H_{\text{gap}} := \sum_{j=1}^{k} \gamma_j\, x_j x_j^T,
\end{equation}
where $k=4000$ and each $x_j$ is a random sparse vector with density $\rho = 10^{-2}$.  
The coefficients $b_j$ are:
\[
\gamma_j =
\begin{cases}
10^2/j^2, & j \in \{1,\ldots,200\}\\
1/j^2, & j \in \{201,\ldots,400\}\\
10^{-2}/j^2, & j \in \{401,\ldots,600\}\\
10^{-6}/j^2, & j \in \{601,\ldots,4000\}
\end{cases}.
\]
The regularization parameter is chosen as $\mu=10^{-6}$; the resulting matrix is ill conditioned, with $\kappa(A+I)\approx  10^8$.

\paragraph{Kernel matrices} \label{subsec:covm}
We consider here matrices that arise from the discretization of kernel functions. The first matrix is obtained by discretizing the radial basis function (RBF) kernel
\begin{equation} \label{eq:cov_RBF}
\tag{Example RBF}
(H_{\text{RBF}})_{i,j}:= \exp\left(-\tfrac{(x_i-x_j)^2}{2\sigma^2}\right),    
\end{equation}
where $x_1,\ldots,x_{n} \sim \mathcal{N}_{0,1}$,  the denominator is $2\sigma^2=10^{-4}$, and the regularization parameter is $\mu=10^{-2}$.
The spectral decay of such a kernel is known, see e.g.~\cite{Belkin18}, and satisfies $\lambda_i(H_{\text{RBF}})=\mathcal{O}(e^{-Ci})$, for some $C>0$. 
The second and third matrices come from discretizations of the Matérn kernel~\cite{Matern60,RasmussenWilliams06}: 
\begin{equation} \label{eq:cov_Matèrn}
\tag{\text{Example Mat\'ern}}
(H^{\nu, \vartheta}_{\text{Mat}})_{i,j}:= \frac{2^{1-\nu}}{\Gamma(\nu)} \left( \frac{\sqrt{2\nu}\lvert x_i -x_j \rvert}{\vartheta}\right)^\nu K_\nu\left( \frac{\sqrt{2\nu}\lvert x_i -x_j \rvert}{\vartheta}\right),
\end{equation}
where $\Gamma$ is the Gamma function, $K_\nu$ is the modified Bessel function, $\vartheta$ is a length scale parameter, $\nu$ is a smoothing parameter and $x_1,\ldots,x_{n} \sim \mathcal{N}_{0,1}$.
The spectral decay of the Matérn kernels has been derived, e.g., in~\cite{SantinSchaback16} and satisfies $\lambda_i(H^{\nu}_{\text{RBF}})=\mathcal{O}(i^{-(2\nu+1)})$.
In the experiments, we consider $\vartheta=1$, and both $\nu=\frac{1}{2}$ with regularization parameter $\mu=10^{-2}$, see e.g.~\cite{FengKulickTang24}, and $\nu=\frac{3}{2}$, with regularization parameter $\mu=10^{-4}$.

\begin{remark}
Since Gaussian random vectors are rotationally invariant, all the methods we considered in this paper are invariant by orthogonal transformation of the matrix $A$; to make the code run faster, we actually work with diagonal matrices.
\end{remark}

Figure~\ref{fig:eigenvalues} shows the spectral decay of the six matrices we consider in the numerical experiments.

\begin{figure}[ht!]
    \centering
    \includegraphics[width=\textwidth]{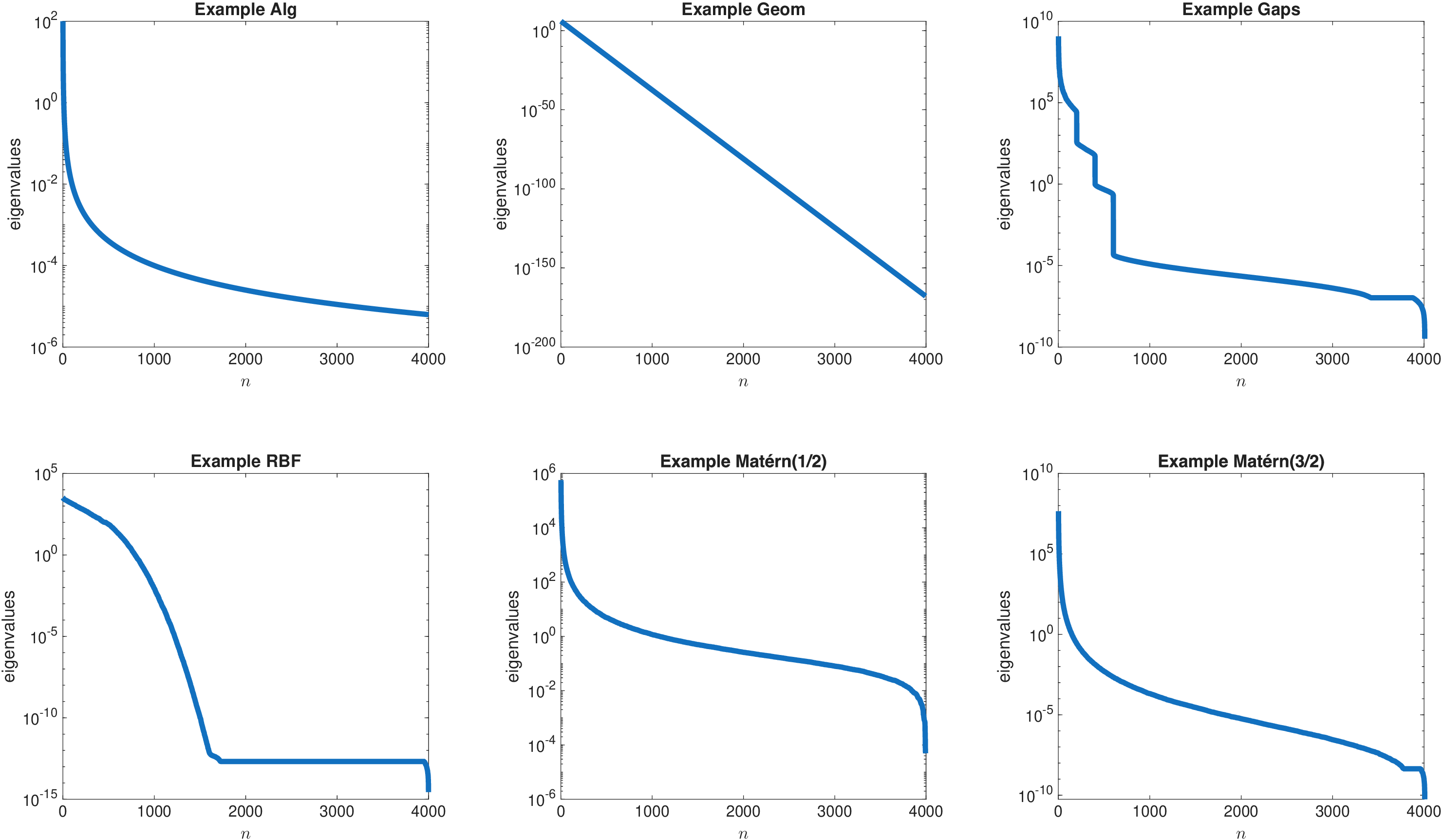}
    \caption{Spectra of the scaled matrices $A\in \RR^{4000\times 4000}$ described in Section~\ref{subsec:matrices}; note that the $y$-axis has a logarithmic scale.}
    \label{fig:eigenvalues}
\end{figure}

\subsection{Tightness of the error bounds from section~\ref{sec:nystrom}} \label{subsec:bounds}
We compare the square root of the theoretical bound~\eqref{eq:exp_logdet_true_fro} and the average error obtained, in practice, by the one-sample strategy~\eqref{eq:onesample-nystrom}, with the bound~\eqref{eq:err:Trunc_Nys_leading} and the average error of the approximation~\eqref{eq:lowrank}, for all our test matrices. 
Given a fixed budget $\ell+m$ of matvecs, the computation of the upper bounds on the error is done by taking the minimum among all the combinations of $k\ge 0$ and $p\ge 2$ with fixed sum; $k+p = \ell$ for  strategy~\eqref{eq:onesample-nystrom} and $k+p = \ell +m$ for strategy~\eqref{eq:lowrank}.

For the same fixed budget of matvecs, we compute~\eqref{eq:onesample-nystrom} and~\eqref{eq:lowrank}. Figure~\ref{fig:comparison1} and Figure~\ref{fig:comparison2} report the results with $m = 50$ and $m = 10$ Lanczos steps, respectively, for the SLQ part. The values of $\ell$ are ranging from $100$ to $1000$. 
The $x$-axis of each plot in Figures~\ref{fig:comparison1} and~\ref{fig:comparison2} represents the number of total matvecs; the dashed lines represent the theoretical bounds -- or its square root, in case of~\eqref{eq:exp_logdet_true_fro}-- divided by the target quantity $\trace \log (A+I)$; the solid lines are computed by averaging the absolute error of the corresponding approximation -- \eqref{eq:onesample-nystrom} or \eqref{eq:lowrank} -- over $100$ runs.

We emphasize that the bounds~\eqref{eq:exp_logdet_true_fro} and~\eqref{eq:err:Trunc_Nys_leading} assume that quadratic forms are computed exactly; therefore, the bound for strategy~\eqref{eq:onesample-nystrom} holds only approximately. When $m=50$, the error of the approximation of the quadratic form is always negligible in our examples, but for $m=10$ the bound~\eqref{eq:exp_logdet_true_fro} starts to hold when the Nyström preconditioner is an effective preconditioner. 
We remark that the bound~\eqref{eq:err:Trunc_Nys_leading} for the~\eqref{eq:lowrank} strategy is slightly tighter than the bound~\eqref{eq:exp_logdet_true_fro} for the~\eqref{eq:onesample-nystrom} strategy; however, they correctly predict when one strategy is better than the other. 
\begin{remark}
The Nyström approximation is implemented as described~\cite{LiLindermanSzlamStantonKlugerTygert17, TroppYurtseverUdellCevher17}; in particular, it is necessary to introduce, for numerical stability, a little shift of the order of $\varepsilon \|A\|_2$. This shift does not allow the accuracy of the approximation to be below this fixed threshold, as clearly visible in~\eqref{eq:example2}.    
\end{remark}

\begin{figure}[ht!]
    \centering
    \includegraphics[scale=.25]{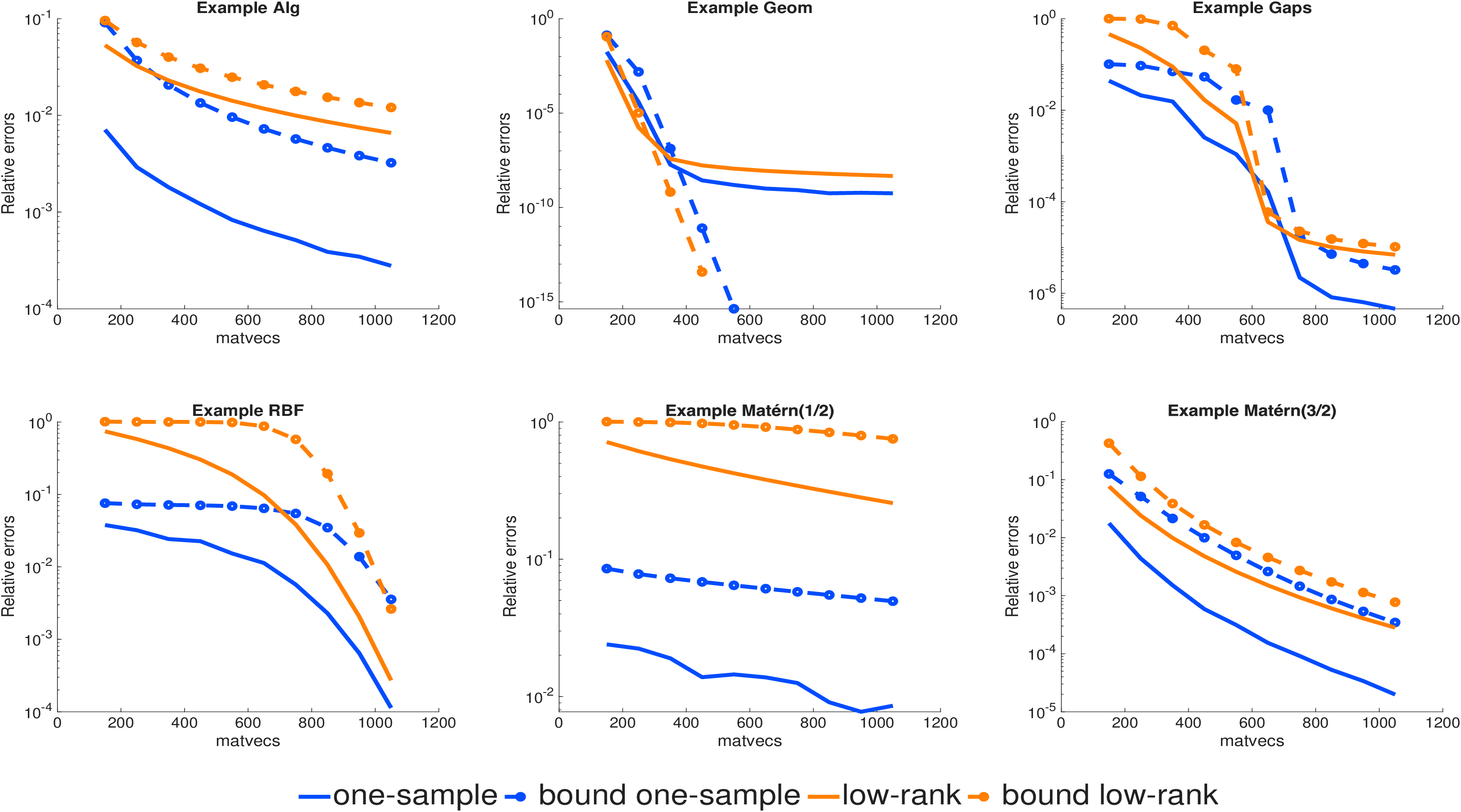}
    \caption{Square root of the upper bound~\eqref{eq:exp_logdet_true_fro} and computed error, on average, for the one-sample strategy~\eqref{eq:onesample-nystrom}, upper bound~\eqref{eq:err:Trunc_Nys_leading} and computed error, on average, for the low-rank strategy~\eqref{eq:lowrank}, considering a Nyström approximation of dimension $\ell=100,\ldots,1000$ and $m=50$ steps of the Lanczos method. Each bound is the optimal among all the combination of $k$ and $p$ with fixed sum, as described in Section~\ref{subsec:bounds}.}
    \label{fig:comparison1}
\end{figure}
\begin{figure}[ht!]
    \centering
\includegraphics[scale=.25]{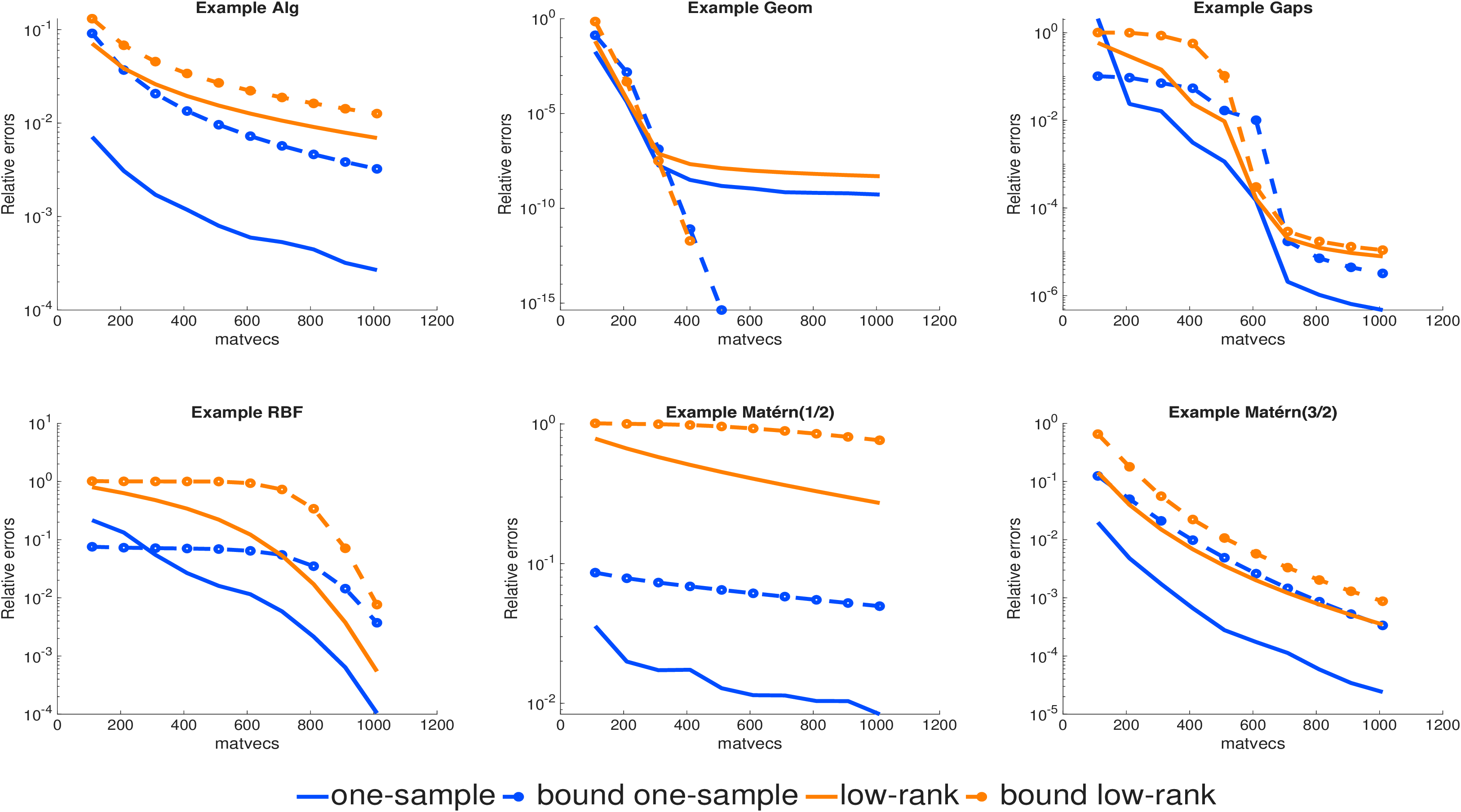}
    \caption{Square root of the upper bound~\eqref{eq:exp_logdet_true_fro} and computed error, on average, for the one-sample strategy~\eqref{eq:onesample-nystrom}, upper bound~\eqref{eq:err:Trunc_Nys_leading} and computed error, on average, for the low-rank strategy~\eqref{eq:lowrank}, considering a Nyström approximation of dimension $\ell=100,\ldots,1000$ and $m=10$ steps of the Lanczos method. Each bound is the optimal among all the combination of $k$ and $p$ with fixed sum, as described in Section~\ref{subsec:bounds}.}
    \label{fig:comparison2}
\end{figure}

\subsection{Convergence of the Lanczos method for quadratic forms with the logarithm} 
\label{subsec:expsLanczos}

Let us illustrate how the errors made in the computation of the quadratic forms influence the overall accuracy of the approximation. We consider, for each of our example matrices, the Lanczos errors 
\begin{equation}\label{eq:lanczos}
\text{Lanc}\left (\hat M_\ell,m,w\right ) = \left \lvert w^T\log\left(\hat M_\ell\right)w-\|w\|^2 e_1^T\log(T_m)e_1\right \rvert,
\end{equation}
defined in~\eqref{eq:errorlanczos}, for $m \in \{10, 50\}$ and values of $\ell$ ranging from $100$ to $1000$; in the plot, the quantity~\eqref{eq:lanczos} is divided by $\trace \log (A+I)$. In Figure~\ref{fig:Lanczos} we fix a Nyström preconditioner $\hat P_\ell$, for each value of $\ell$, and we plot the average error over $100$ choices of the Gaussian random vector $w$ for the fixed preconditioned matrix $\hat M_\ell$. In each plot, we also show the error of the trace estimator
\begin{equation} \label{eq:baderror}
\text{tracest}_\ell(w) := \left \lvert \trace\log\left(\hat P_\ell^{-\frac{1}{2}}(A+I)\hat P_\ell^{-\frac{1}{2}}\right)- w^T \log\left(\hat P_\ell^{-\frac{1}{2}}(A+I)\hat P_\ell^{-\frac{1}{2}}\right) w \right \rvert ,   
\end{equation}
averaged over the same $100$ random Gaussian vectors $w$, and divided by the target quantity $\trace\log(A+I)$. The strategy~\eqref{eq:onesample-nystrom} has two sources of error: the Lanczos approximation of the quadratic form~\eqref{eq:lanczos} and the trace estimation error~\eqref{eq:baderror}. In Figure~\ref{fig:Lanczos} we see that, even for moderately small values of $\ell$, the Lanczos approximation with $m = 10$ is good enough in the sense that~\eqref{eq:lanczos} is negligible with respect to~\eqref{eq:baderror}, and it does not make sense to pursue a better approximation of the quadratic form. 

\begin{figure}[ht!]
    \centering \includegraphics[width=\textwidth]{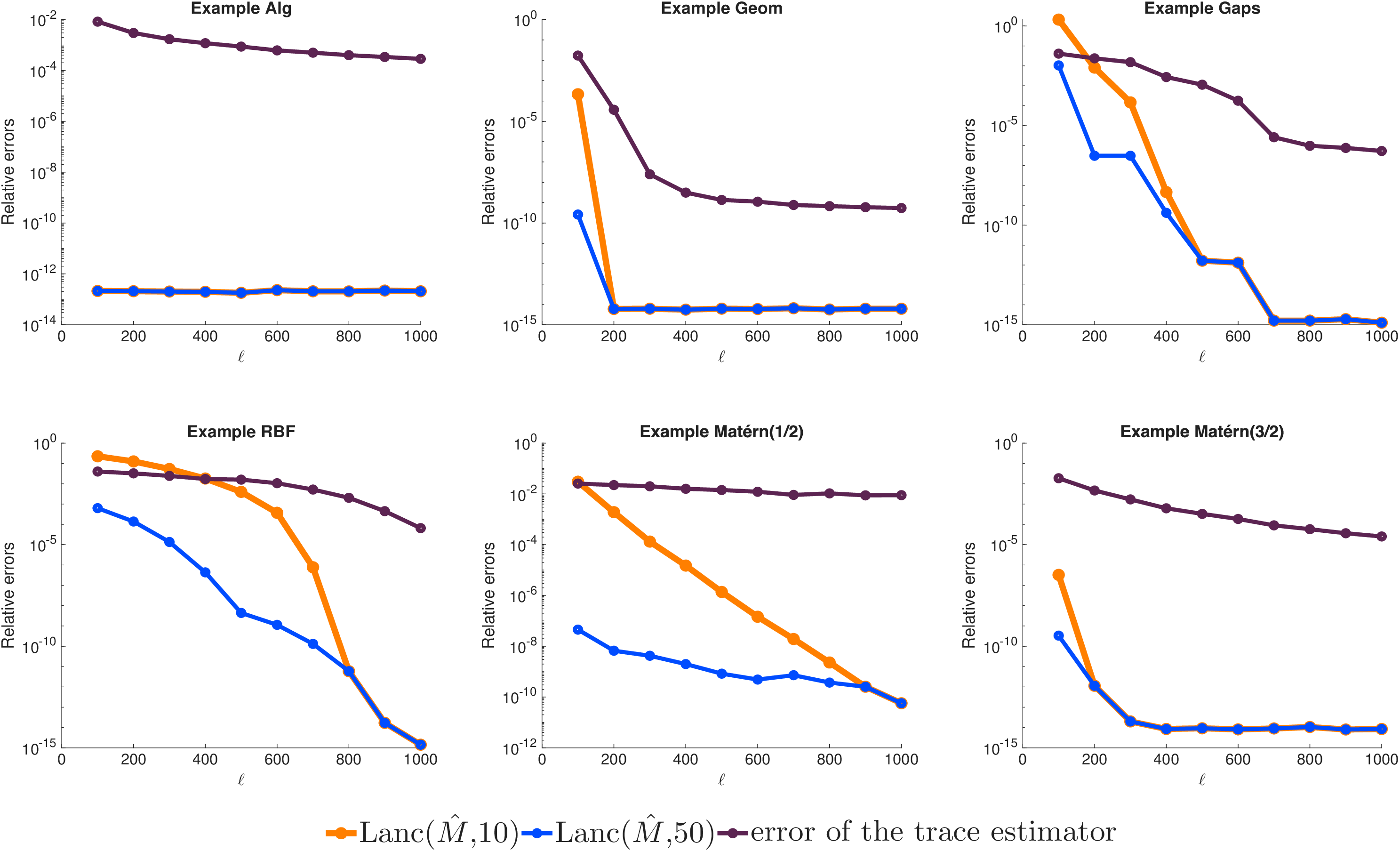}
    \caption{Plot of the error of the Lanczos method~\eqref{eq:lanczos}, averaging on $100$ Gaussian samples, divided by the target quantity $\trace\log(A+I)$ after $m=10$ and $m=50$ steps and comparison with the error of the trace estimator~\eqref{eq:baderror}, averaging on the same Gaussian samples and dividing by the target quantity $\trace\log(A+I)$, considering an approximation of dimension $\ell=100,\ldots,1000$ as described in Section~\ref{subsec:expsLanczos}}.
    \label{fig:Lanczos}
\end{figure}

\subsection{Choice of the parameter \texorpdfstring{$\beta$}{beta}} \label{subsec:expsbetas}
We show the behavior of the log-det-ective method (Algorithm~\ref{alg:logdetective}) for different values of the initial budget allocated for the low-rank approximation of the matrix, represented by the value of $\beta$, while the number of steps used to approximate any quadratic form in Lanczos method is fixed to $m=10$. In the experiments we take $\beta \in \left \{\frac 18, \frac 14, \frac 12,\frac 34, \frac 78\right \}$ and values of $\ell$ ranging from $100$ to $1000$. 
In each plot, the $x$-axis is the total number of matvecs $\ell+m$ and the $y$-axis is the final relative error in the approximation of the log-determinant. For each value of $\ell$ that we consider, we run each algorithm $100$ times; the plots report the mean relative error,
with respect to the target quantity $\trace \log(A+I)$, and the standard deviation obtained discarding the best and the worst $10\%$ errors over the runs, represented by vertical bars.
{To make the bars relative to the standard deviation visually distinguishable among the methods, some lines are slightly shifted to the left and to the right, although the budget of matvecs used for different methods is still the same.}
The plots in Figure~\ref{fig:betas} show that if we choose a value of $\beta$ that is too small, the method can fail to detect when it is convenient to switch the strategy, and this is due to the fact that the rank-$\beta \ell$-approximation is very different from the rank-$\ell$ approximation.
Specifically, in~\eqref{eq:example1}, and~\eqref{eq:cov_Matèrn} with $\nu=\frac{1}{2}$, Algorithm~\ref{alg:logdetective} with the values $\beta=\frac 18$, $\beta= \frac{1}{4}$ in~\eqref{eq:cov_RBF} wrongly applies the one-sample strategy. Besides, in~\eqref{eq:cov_RBF}, values of $\beta=\frac 18$, $\beta= \frac 14$ and $\beta=\frac 12$ wrongly predicts the mixed strategy.
From now on we fix $\beta=\frac{3}{4}$.

\begin{figure}[ht]
    \centering
    \includegraphics[width=\textwidth]{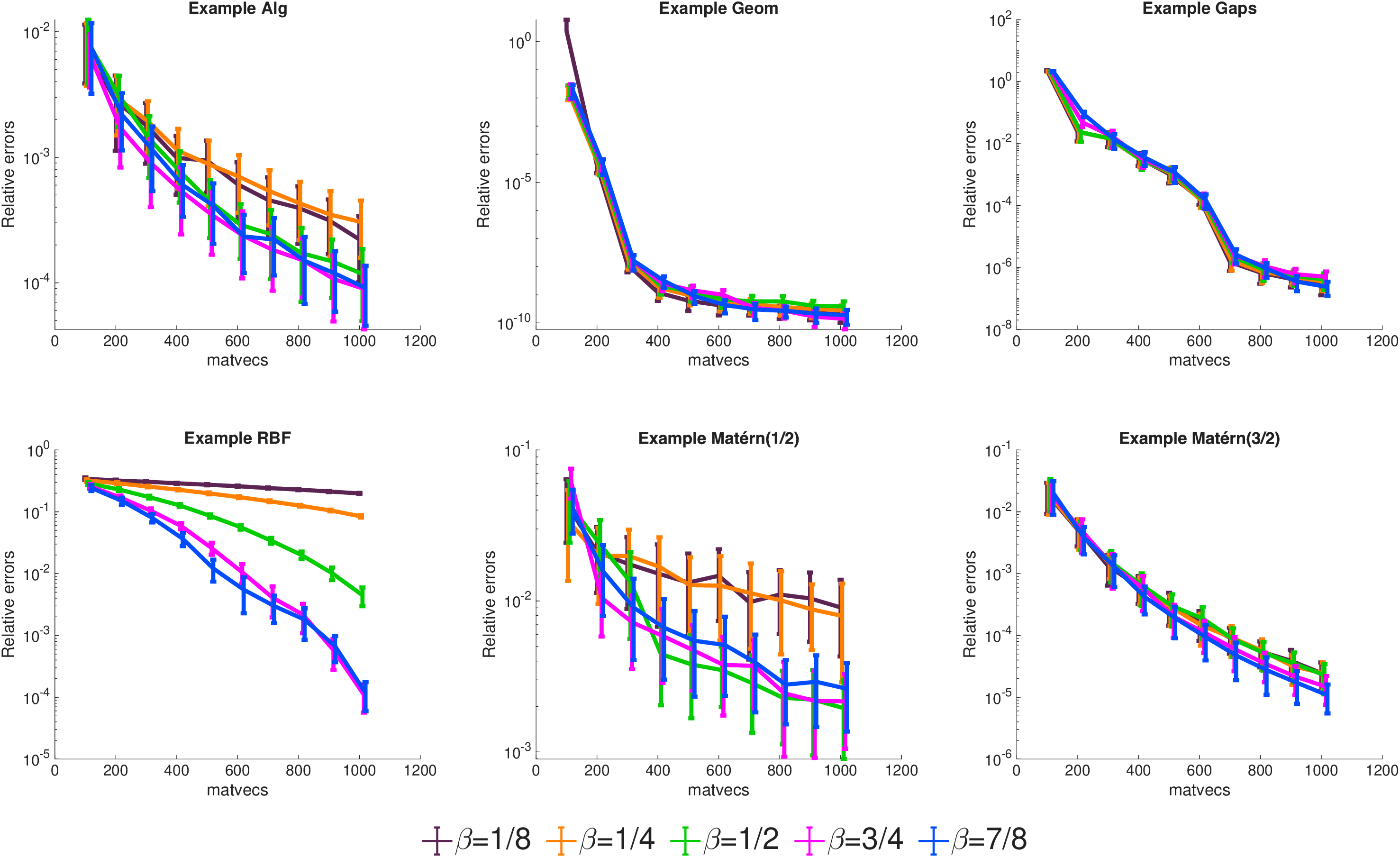}
    \caption{Comparison of the error of the log-det-ective strategy of Section~\ref{sec:detective_strategy} for $\beta \in \{\frac 18, \frac 14, \frac 12,\frac 34, \frac 78\}$, considering a Nyström approximation of dimension $\ell=100,\ldots,1000$ and $m=10$ steps of the Lanczos method, averaging on $100$ runs.}
    \label{fig:betas}
\end{figure}

\subsection{Comparison with existing strategies} \label{subsec:methods}
For a fixed budget of matvecs, we compare our one-sample strategy~\eqref{eq:onesample-nystrom} and log-det-ective strategy (Algorithm~\ref{alg:logdetective}) with several existing techniques that are tailored to the computation of traces of matrix functions:
\begin{itemize}
    \item plain stochastic Lanczos quadrature~\eqref{eq:gh};
    \item the funNys++ algorithm in~\cite{PerssonKressner23};
    \item the Krylov-aware method, corresponding to~\cite[Algorithm 3.1]{ChenHallman23};
    \item the low-rank strategy~\eqref{eq:lowrank};
    \item the strategy~\eqref{eq:mixed_strategy-nystrom} which splits the number of matvecs evenly between the creation of the preconditioner. This is called \say{\emph{half samples}} in the plots and is inspired by~\cite{WengerPleissHennigCunninghamGardner22}.
\end{itemize}
Motivated by the results in Section~\ref{subsec:expsLanczos}, we fix the number of iterations of the Lanczos method as $m=10$ in all the methods where a preconditioner (either in an implicit or explicit way) is used. In particular, for the Krylov-aware strategy~\cite{ChenHallman23} we need to set a block size $b$, a deflation parameter $q$ which represents the number of blocks used for the implicit preconditioner, and a number of samples $N$ to estimate the (implicit) residual. In our experiment, we set $b=4$, $N=1$, and $q=\max\left\{0,\left \lfloor\frac{(\ell +m) - (b+N) m}{b}\right \rfloor\right\}$, so that we use at least $\ell + m$ matvecs, which is our budget. The rationale behind this choice is the same used for our proposed strategies: we are (implicitly) preconditioning as much as we can and use few matvecs to estimate the (implicit) residual. In this sense, the Krylov-aware technique can be seen as a different preconditioning strategy. When running the algorithm, in some examples we stop before the budget reaches $1010$ matvecs because the deflated space becomes numerically rank-deficient. %we possibly have to stop the budget $\ell+m$ before $1010$ whenever the deflated space is rank deficient; the gray line in the plot shows what happens when we ignore the rank deficiency and we just continue running the algorithm.  

When comparing with~\eqref{eq:gh}{ and funNys++}, however, we should run the Lanczos algorithm for a potentially larger number of steps $\hat m$; we choose the smallest $\hat m $ such that
\begin{equation} \label{eq:hatm_gh}
\text{Lanc}\left (A+I,\hat m,w\right) \le \frac{m}{\ell+m}\|\log(A+I)\|_F,
\end{equation}
to make sure that the error incurred in the computation of quadratic forms is smaller, or of the same order of magnitude, as the trace estimation error. Of course, this choice can only be made for illustration purposes: in practice, computing $\|\log(A+I)\|^2_F$ is as expensive as computing the quantity $\|\log(A+I)\|_*$, in which we are interested.
%which should satisfy  
% \begin{equation*}
% \text{Lanc}\left (A+I,\hat m,w\right) \le \frac{\hat m}{\ell+m}\|\log(A+I)\|_F;
% \end{equation*} indeed, this choice ensures that the square of the error caused by Lanczos over all the $\left \lfloor \frac{\ell+m}{\hat m} \right \rfloor$ samples is
% \begin{equation*}
%     \sum_{i=1}^{\left \lfloor \frac{\ell+m}{\hat m} \right \rfloor} \text{Lanc}\left (A+I,\hat m,w_i\right)^2 \approx \left (\frac{\ell +m}{\hat m} \right)\text{Lanc}\left (A+I,\hat m,w\right)^2 \le \left (\frac{\hat m}{\ell +m}  \right)\left \lVert \log(A+I) \right \rVert^2_F
% \end{equation*}
% and thus it is approximately half of the variance caused by the stochastic trace estimator using $\left \lfloor \frac{\ell+m}{\hat m} \right \rfloor$ Gaussian samples.
% Checking this condition requires the solution of a nonlinear equation, which is beyond the scope of the work. In the experiments, we simply check whenever
% \begin{equation*}
% \text{Lanc}\left (A+I,\hat m,w\right) \le \frac{m}{\ell+m}\|\log(A+I)\|_F.
% \end{equation*}
% This is a stronger constraint until $10=m \le \hat m$; hence we choose final $\tilde m:=\max\{m,\hat m\} $.
% If not for comparison purposes, this constraint is unfeasible too, since it requires the knowledge of $\|\log(A+I)\|^2_F$, which is in principle not simpler to obtain than the quantity $\|\log(A+I)\|_*$, in which we are interested. 

As before, for each value of $\ell$ that we consider, we run each algorithm $100$ times; the plots in Figure~\ref{fig:comparison_methods} report the mean relative error, with respect to the target quantity $\trace \log(A+I)$, and the standard deviation obtained discarding the best and the worst $10\%$ errors over the runs. The $x$-axis is the total number of matvecs, the $y$-axis is the final relative error in the computation of the log-determinant and the vertical bars denote the standard deviation. 
Again for clarity purposes, the lines related to the proposed strategies are thicker, and others are slightly shifted to the left and to the right, although the budget of matvecs used for different methods is still the same.
Figure~\ref{fig:comparison_methods} confirms that the one-sample strategy~\eqref{eq:onesample-nystrom} is almost always close to optimal, and generally better than splitting the budget equally (the ``half-samples'' strategy). As expected, SLQ is slow to converge since it does not exploit any low-rank structure in the matrix $A$; on the other side, the strategy~\eqref{eq:lowrank} does not take into account the residual an also performs rather suboptimally in most examples. As the Krylov-aware algorithm is also building an implicit preconditioner and we chose the parameters specifically to mimic a large preconditioner and a small budget for the implicit residual, its performance is similar to our proposed methods for larger values of the budget, while it sometimes has a slower start. 

Finally, the log-det-ective method (Algorithm~\ref{alg:logdetective}) is very effective, and it is able to correctly identify the cases in which the one-sample strategy is suboptimal.  %, which corrects the~\eqref{eq:onesample-nystrom} whenever it is not the (near) optimal one. 
%The plots in Figure~\ref{fig:comparison_methods} show that the low-rank approximation~\eqref{eq:lowrank} stops from being the best one whenever the Lanczos method converges better than the trace estimator, which starts to hold already when the low-rank approximation is small.

\begin{figure}[ht]
    \centering
    \includegraphics[width=\textwidth]%{methods/all_all_methods10.eps}
    {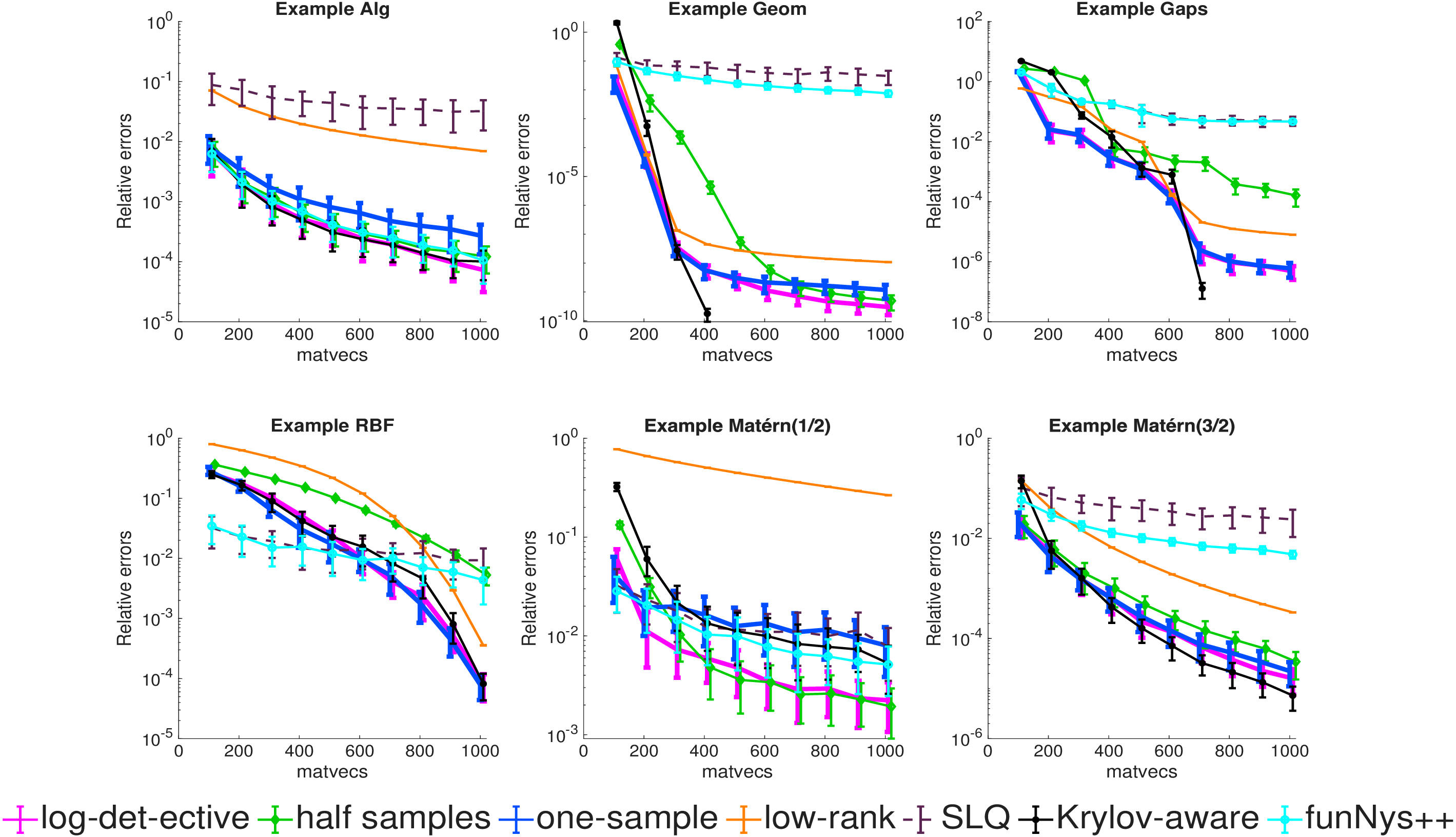}
    \caption{Comparison of the errors of the strategies listed in Section~\ref{subsec:methods}, considering a Nyström approximation of dimension $\ell=100,\ldots,1000$ and $m=10$ steps of the Lanczos method for the methods involving a preconditioner and number of steps $\hat m$ that satisfies~\eqref{eq:hatm_gh} for the strategy~\eqref{eq:gh} and funNys++, averaging on $100$ runs.} \label{fig:comparison_methods}
\end{figure}

\subsection{Comparison for matrices of increasing sizes}\label{subsec:large_examples}
We compare the behavior of our one-sample strategy~\eqref{eq:onesample-nystrom} and log-det-ective method (Algorithm~\ref{alg:logdetective})  with the \say{half samples} strategy recalled in Section~\ref{subsec:methods} and the low-rank strategy~\eqref{eq:lowrank} considering kernel matrices of increasing size $n$ ranging from $1000$ to $10000$ and a budget of total matvecs for the approximation of the log-determinant equal to $\frac{n}{10}+m$, for $m=20$.
We consider a $d$-dimensional RBF and a $d$-dimensional Matérn kernel, defined by the expressions~\eqref{eq:cov_RBF} and~\eqref{eq:cov_Matèrn} respectively, in which we use $d$-dimensional Gaussian vectors $X_i \sim \mathcal{N}_{0,I_d}$ instead of $x_i \sim \mathcal{N}_{0,1}$, and the squared Euclidean norm of $X_i - X_j$ takes the place of the square of $x_i - x_j$. 
We take a $5$-dimensional RBF kernel with $\sigma^2 = 5/2$ and $\mu = 10^{-2}$ and a $3$-dimensional Matérn kernel with $\vartheta=1$, $\nu = \frac{5}{2}$ and $\mu = 10^{-2}$.
For each value of $n$ that we consider, we run each algorithm $100$ times; the plots in Figure~\ref{fig:big_methods} report the mean relative error, with respect to the target quantity $\trace \log(A+I)$, and the standard deviation obtained discarding the best and the worst $10\%$ errors over the runs. The $x$-axis is the size of the kernel matrix, the $y$-axis is the final relative error in the computation of the log-determinant and the vertical bars denote the standard deviation. 
Figure~\ref{fig:big_methods} shows the effectiveness of the proposed methods when increasing the dimension of the problem. 
\begin{figure}
    \centering
    \includegraphics[scale=.35]{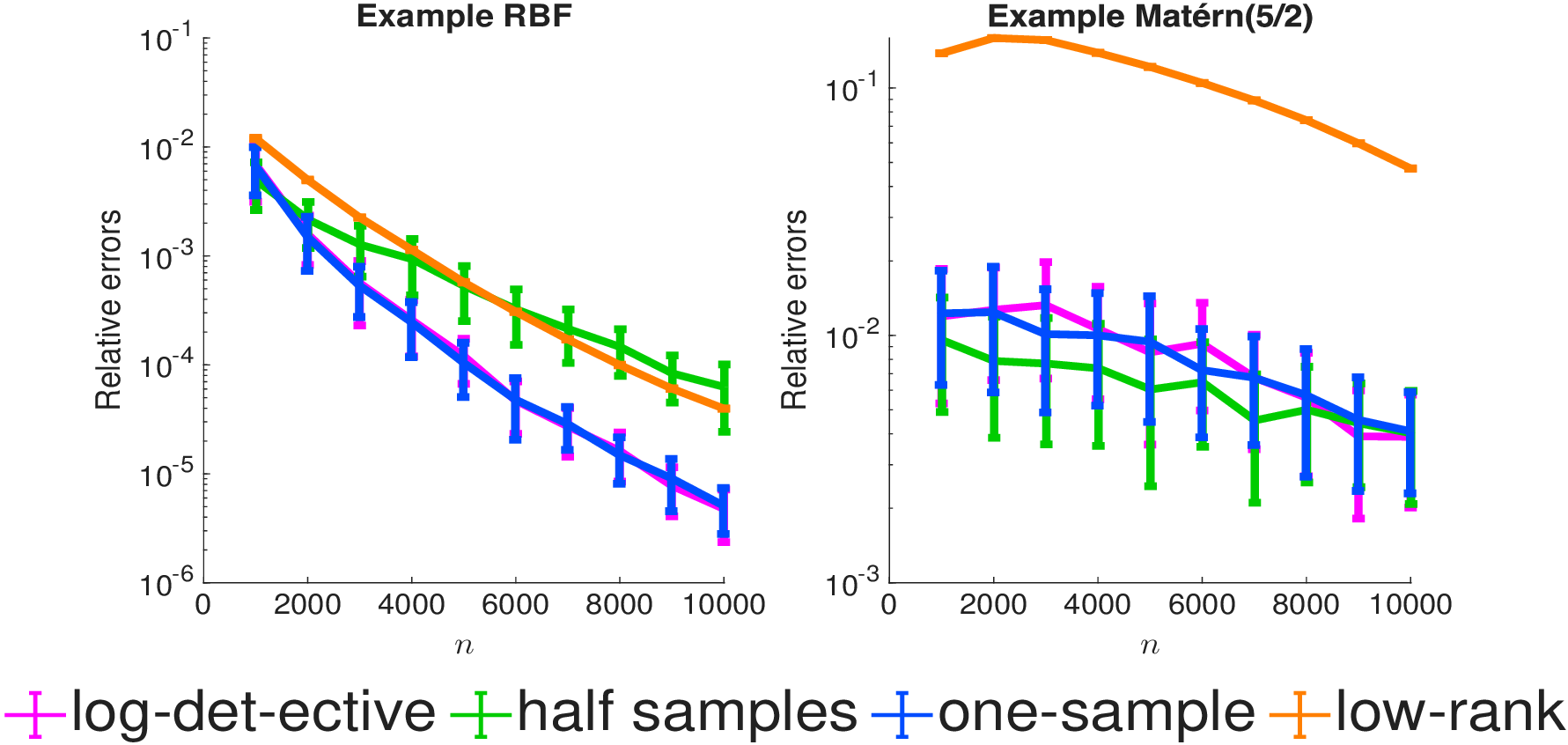}
    \caption{Comparison of the errors of the strategies listed in Section~\ref{subsec:large_examples}, considering kernel matrices of increasing size $n=1000,\ldots,10000$, a budget of matvecs for the approximation of the log-determinant equal to $n/10+m$ and $m=20$ steps of the Lanczos method, averaging on $100$ runs.} \label{fig:big_methods}
\end{figure}
\section{Conclusions}

We proposed two variance reduction methods to approximate the regularized log-determinant of an SPSD matrix $A$, which combine preconditioning and the use of few samples for estimating the residual term via SLQ. We used a generalized Nystr\"om preconditioner, which is cheap to compute and effective for matrices with some spectral decay. We theoretically analyzed the method in the case when we pick one single sample for SLQ, and we tested this strategy on matrices exhibiting a variety of eigenvalue decays, showing numerically that this choice is often the optimal one. 
{For matrices with moderate spectral decay, we observe that, although suboptimal, it still comparable to the best methods, and thus it is a safe simple choice.   
To automatically detect the suboptimal situations, we 
proposed a \say{log-det-ective} strategy (Algorithm~\ref{alg:logdetective}) which, in these cases,} changes the allocation of the budget of matvecs, allowing more vectors for the SLQ part.
Several numerical experiments justify the choice of parameters for the \say{log-det-ective} method and confirmed the effectiveness of both strategies.

From a practical point of view, some technical modifications can be made: e.g., choosing a different random embedding matrix for the Nyström approximation, choosing other preconditioners, or taking differently distributed random vectors for the trace estimation part, for instance Rademacher vectors or uniform random vectors on a sphere. 

\subsection*{Acknowledgments} We thank Michele Benzi, Stefano Massei, and Valeria Simoncini for helpful discussions on topics related to this work. Both authors are members of INdAM / GNCS.

%%%%%%%%%%%%%%%%%%%%%%%%%%%%%%%%%%%%%%%%%

\bibliographystyle{plain}
\bibliography{matscinet_biblio}

\end{document}